\documentclass[reqno,12pt]{amsart}
\usepackage{geometry}
\usepackage{mathtools,amssymb,amsthm,mathrsfs,color,lineno,paralist,graphicx,float}
\usepackage[colorlinks,
linkcolor=blue,
anchorcolor=green,
citecolor=blue, 
]{hyperref}
\usepackage{tikz}
\usepackage{enumitem}
\usepackage[T1]{fontenc}

\usepackage{subfig} 
\usepackage[justification = centering, labelsep =period]{caption} 

\usepackage{calc}
\newtheorem{theorem}{Theorem}[section]
 
\newtheorem{lemma}{Lemma}[section]
\newtheorem{remark}{Remark}[section]

\allowdisplaybreaks[2]
\title[Polynomial Spectral Gap Estimates]{Polynomial Spectral Gap Estimates for Finitely Differentiable Quasi-Periodic Schr\"odinger Operators}

\author{Ao Cai}
\address{School of Mathematical Sciences, Soochow University, Suzhou, China}
\email{acai@suda.edu.cn}

\author{Yang Gao}
\address{School of Mathematical Sciences, Soochow University, Suzhou, China}
\email{glacier189@163.com}

\author{Jiawei He$^{*}$}
\thanks{$^{*}$Corresponding author.}
\address{Fujian Key Laboratory of Financial Information Processing, Putian University\\
Putian, Fujian 351100, P.R. China}
\email{hermit\_well@163.com}

\date{}

\begin{document}

\begin{abstract}
We prove polynomial upper bounds on spectral gaps of quasi-periodic Schr\"odinger operators with small $C^{k}$ potentials and Diophantine frequencies. For sufficiently large $k$, the decay exponent is $k-34\tau$, where $\tau$ is the Diophantine exponent. We construct examples with lower bounds along infinitely many gap labels, showing that the leading coefficient $1$ of $k$ is asymptotically optimal. The proof uses quantitative reducibility and the Moser-P\"oschel argument.
\end{abstract}

\maketitle

\section{Introduction}

\subsection{Motivation and  main results}
We consider the quasi-periodic Schr\"odinger operator on $\ell^{2}(\mathbb{Z})$ defined by
\begin{equation*}
    (H_{V,\alpha,\theta}u)_{n}=u_{n+1}+u_{n-1}+V(\theta+n\alpha)u_{n},\qquad n\in\mathbb{Z},
\end{equation*}
where $\theta\in\mathbb{T}^{d}\coloneqq\mathbb{R}^{d}/\mathbb{Z}^{d}$ is the phase, $\alpha\in\mathbb{R}^{d}$ is the frequency, and $V\in C^{k}(\mathbb{T}^{d},\mathbb{R})$ is the potential. Throughout the paper, we assume that $1,\alpha_{1},\ldots,\alpha_{d}$ are rationally independent. The operator $H_{V,\alpha,\theta}$ is bounded and self-adjoint, and its spectrum is a compact perfect subset of $\mathbb{R}$ independent of $\theta$. We denote this common spectrum by $\Sigma_{V,\alpha}$.

A fundamental example is the almost Mathieu operator, corresponding to $d=1$ and $V(\theta)=2\lambda\cos(2\pi\theta)$:
\begin{equation*}
    (H_{\lambda,\alpha,\theta}u)_{n}=u_{n+1}+u_{n-1}+2\lambda\cos\big(2\pi(\theta+n\alpha)\big)u_{n},\qquad n\in\mathbb{Z},
\end{equation*}
where $\lambda\in\mathbb{R}$ is the coupling constant. Quasi-periodic Schr\"odinger operators have attracted considerable attention because of their connections with quasicrystals and the quantum Hall effect, see \cite{AvilaA15,Avila2009TheTM,AvilaJ10,Artury21,Simonb00,HouY12,You18}. The
richness of its spectral theory, the abundance of unusual features, and its amenability to analysis also make it a crucial component of general
conjectures and counterexamples in mathematics.  For a systematic treatment of the spectral theory of one-dimensional ergodic Schr\"odinger operators, we refer to the two-volume monograph of Damanik and Fillman \cite{DamanikFillman22,DamanikFillman24}. 

In this paper, we study the sizes of spectral gaps for small finitely differentiable quasi-periodic potentials. Our aim is to determine how the polynomial gap-decay exponent depends on the differentiability order $k$. We address this question through a uniform upper bound and lower bounds for explicitly constructed potentials.
A spectral gap is a bounded connected component of $\mathbb{R}\setminus\Sigma_{V,\alpha}$. By the gap-labelling theorem \cite{DamanikF22}, each gap has a unique label $m\in\mathbb{Z}^{d}\setminus\{0\}$. We use the labelling convention recalled in \hyperref[pre]{Section~\ref*{pre}} and write
\begin{equation*}
    G_{m}(V)\coloneqq(E_{m}^{-},E_{m}^{+}),\qquad |G_{m}(V)|\coloneqq E_{m}^{+}-E_{m}^{-},
\end{equation*}
allowing $E_{m}^{-}=E_{m}^{+}$ when the gap is collapsed.
For $\gamma>0$ and $\tau>d$, we write $\alpha\in\mathrm{DC}_{d}(\gamma,\tau)$ if
\begin{equation*}
    \inf_{j\in\mathbb{Z}}|\langle m,\alpha\rangle-j|\geqslant\frac{\gamma}{|m|^{\tau}},\qquad m\in\mathbb{Z}^{d}\setminus\{0\},
\end{equation*}
where $|m|\coloneqq|m_{1}|+\cdots+|m_{d}|$.

We now state our main result.  

\begin{theorem}\label{GYGJMGYQ}
 Fix $\gamma>0$, $\tau>d$, $0<\sigma<\frac{1}{6}$, and an integer $k>(\frac{30}{\sigma}+2)\tau+2$. Let $\alpha\in\mathrm{DC}_{d}(\gamma,\tau)$. Choose $D=D(k,\tau,\sigma)$ such that
\begin{equation*}
    \max\bigg\{\frac{30}{\sigma},\frac{k-3}{\tau}-2\bigg\}<D<\frac{k-2}{\tau}-2.
\end{equation*} 
Then there exists $\varepsilon=\varepsilon(\sigma,\gamma,\tau,k,d)>0$ with the following properties.
\begin{enumerate}[label=(\arabic*)]
\item For every $V\in C^{k}(\mathbb{T}^{d},\mathbb{R})$ with $\|V\|_{k}\leqslant\varepsilon$ and every $m\in\mathbb{Z}^{d}\setminus\{0\}$,
\begin{equation*}
    |G_{m}(V)|\leqslant\varepsilon^{\frac{1}{6(D-5)}}|m|^{-(k-34\tau)}.
\end{equation*}

\item If $k>361\tau$, then for every $\delta\in(0,1)$ and $\varepsilon_{0}\in(0,\varepsilon)$ there exist $V\in C^{k}(\mathbb{T}^{d},\mathbb{R})$ with $\|V\|_{k}\leqslant\varepsilon_{0}$ and a sequence $m_{i}\in\mathbb{Z}^{d}\setminus\{0\}$ with $|m_{i}|\to\infty$ such that
\begin{equation*}
    |G_{m_{i}}(V)|\geqslant|m_{i}|^{-(1+\delta)k-6\tau}.
\end{equation*}
\end{enumerate}
\end{theorem}

Compared with the upper bound in \cite{CAI2021109035}, part (1) of \hyperref[GYGJMGYQ]{Theorem~\ref*{GYGJMGYQ}} improves the decay exponent from $k/9$ to $k-34\tau$. In particular, the loss in the decay exponent is bounded independently of $k$.
Part (2) constructs arbitrarily small $C^{k}$ potentials whose gap lengths satisfy polynomial lower bounds with exponent $(1+\delta)k+6\tau$ along an infinite sequence of labels, for every $\delta>0$. Together, the two parts determine the optimal leading coefficient of $k$ in uniform polynomial gap bounds as $k\to\infty$ with $\tau$ fixed.

\subsection{Reducibility}
Reducibility provides a dynamical approach to the spectral theory of quasi-periodic Schr\"odinger operators. In the analytic setting, Eliasson used a KAM scheme to prove reducibility for almost every energy for quasi-periodic Schr\"odinger equations with small potentials and Diophantine frequencies \cite{Eliasson92}. Puig obtained a nonperturbative extension in the one-frequency setting, with a smallness threshold independent of the Diophantine constants \cite{Puig06}. Avila and Jitomirskaya used quantitative Aubry duality to establish almost reducibility for Schr\"odinger cocycles with small analytic potentials and Diophantine frequencies \cite{AvilaJ10}. Hou and You proved almost reducibility for analytic  linear systems  on all irrational frequency and obtained reducibility under arithmetic conditions on the rotation number \cite{HouY12}.
 Avila's almost reducibility conjecture \cite{Artura10,Artura23} extends the local theory to the whole subcritical regime: every subcritical cocycle is almost reducible. For Schr\"odinger cocycles with trigonometric polynomial potentials and non-exponentially approximated frequencies, independent proofs were given by Ge \cite{Ge24} and Han and Schlag \cite{HanSchlag26}.

For finitely differentiable cocycles, quantitative almost reducibility and reducibility have been used to study spectral continuity, absolutely continuous spectrum, localization and transport, see \cite{cai2021absolutely,cai17,Cai2017ReducibilityOF,CaiLvWang2025,Shan25}. A main issue is the loss of regularity in the conjugacies. For Diophantine frequencies, Cai, Chavaudret, You and Zhou established quantitative $C^{k,k'}$ almost reducibility with $k'\leqslant k/6$ and used it to prove $\frac{1}{2}$-H\"older continuity of the Lyapunov exponent and the integrated density of states \cite{cai17}. Cai improved the retained regularity to $k'\leqslant k-2\tau-2$ and used the refined estimates to obtain purely absolutely continuous spectrum for sufficiently small finitely differentiable potentials \cite{cai2021absolutely}.

Under the additional assumption that the rotation number is Diophantine or rational with respect to the frequency, Cai and Ge proved reducibility for sufficiently small $C^{k}$ perturbations of constant cocycles, retaining $k'=\lfloor k/400\rfloor$ derivatives \cite{Cai2017ReducibilityOF}. Under the same condition on the rotation number, Cai, Lv and Wang established reducibility for $k>14\tau+2$ with $k'<k-10\tau-3$ \cite{CaiLvWang2025}. The loss of regularity can therefore be bounded independently of $k$.

Applications to localization require quantitative bounds for the conjugacies as well as control of their regularity. Shan developed estimates adapted to polynomial dynamical localization for finitely differentiable quasi-periodic long-range operators \cite{Shan25}. More recently, Cai, Lv, Shan and Wang established arithmetic polynomial spectral and dynamical localization through refined quantitative $C^{k}$ reducibility estimates \cite{cai2026arithmetic}. For results on reducibility, almost reducibility and their applications in other non-analytic settings, see \cite{Chengge21,Claire13,HouS19,WangYou26}.

\subsection{Spectral gaps}
Precise spectral gap estimates describe how gap lengths depend on their labels and help us understand the physical and mathematical properties of quantum systems. Upper bounds control the decay of gap lengths and are useful in the study of spectral homogeneity and inverse spectral problems \cite{CAI2021109035,zhaozhiyan}. Positive lower bounds show that the corresponding gaps are open.
Gap opening is also related to the topological structure of the spectrum. If the endpoints of open gaps are dense in the spectrum, the spectrum is a Cantor set, that is, a compact, perfect and nowhere dense set. For the almost Mathieu operator, the Ten Martini problem asks whether its spectrum is a Cantor set for every irrational frequency and every nonzero coupling. The stronger Dry Ten Martini problem asks whether every gap allowed by the gap-labelling theorem is open. For results on Cantor spectrum and the Ten Martini problems, see \cite{Avila2009TheTM,AvilaYZ,GeJitYou23,GeJitYou26,HeCheng2025,HHSY,HouS19,puig}.

For gap estimates, one needs quantitative bounds for the conjugacy and the reduced constant in terms of the gap label. 
The Moser-P\"oschel argument connects gap opening and gap estimates with quantitative reducibility \cite{Moser84,Puig06}. Leguil, You, Zhao and Zhou established exponential upper and lower bounds for the gaps of noncritical almost Mathieu operators with Diophantine frequencies \cite{zhaozhiyan}. For general sufficiently small analytic potentials with Diophantine frequencies, they also obtained exponential upper bounds with decay rates arbitrarily close to the width of analyticity. 
Liu and Shi obtained exponential upper bounds for the spectral gaps of weakly coupled analytic quasi-periodic Schr\"odinger operators with Liouville frequencies \cite{LiuS19}. 
In the finitely differentiable setting, Cai and Wang proved a polynomial upper bound with decay exponent $k/9$ for sufficiently small $C^{k}$ potentials with Diophantine frequencies \cite{CAI2021109035}.
For further results on spectral gaps and gap opening, see \cite{Krasovsky17,Moser84,Puig06,Shiyuan19}.

Explicit constructions provide a way to obtain quantitative lower bounds. In the continuous setting, He, Hou, Shan and You constructed quasi-periodic Schr\"odinger operators with analytic potentials and Cantor spectrum \cite{HHSY}. Hou, Shan and You gave a construction for Gevrey potentials based on KAM iterations and gap opening \cite{HouS19}. For finitely differentiable discrete operators, He and Cheng obtained polynomial upper and lower bounds along a selected set of gap labels by controlling the conjugacies and the nonzero parabolic shear in a quantitative reducibility scheme \cite{HeCheng2025}.

\subsection{Main ideas of the proof}
For the upper bound, the main task is to control the dependence of both the parabolic shear and the conjugacy on the gap label. We refine the quantitative estimates in the reducibility scheme so that the shear retains a decay exponent with leading coefficient $1$ in $k$. At the same time, the conjugacy at a fixed regularity grows with the gap label at a rate independent of $k$. These two estimates are needed together to obtain the improved gap bound.

A further point is the choice of the parameter in the Moser-P\"oschel argument. A fixed parameter would lose a fixed proportion of the decay exponent, even with the improved shear estimate. We instead let this parameter depend on $k$ and verify the corresponding smallness condition using the label-dependent conjugacy bound. This limits the loss in the gap-decay exponent to a quantity independent of $k$ and gives the exponent $k-34\tau$.

For the lower bound, we adapt the sparse Fourier construction in \cite{HeCheng2025}. The Fourier coefficients are chosen with a small additional decay to ensure that the potential belongs to $C^{k}$ and can have arbitrarily small norm. The main estimate shows that the selected resonant coefficient remains larger than the total error throughout the iteration. We also control the limiting conjugacy so that passing from the parabolic shear to the gap length introduces only an additive loss independent of $k$. The resulting lower bounds have leading coefficients arbitrarily close to $1$. Together with the uniform upper bound, they establish the optimal leading dependence of the gap-decay exponent on $k$.

\subsection{Organization of the paper}

Section~2 collects the notation and preliminary results. In Section~3, we establish quantitative reducibility, prove part (1) of \hyperref[GYGJMGYQ]{Theorem~\ref*{GYGJMGYQ}}, and derive the lower-gap estimate needed for part (2). Section~4 introduces the potential used for part (2) and the corresponding iteration. In Section~5, we deduce part (2) from a quantitative reducibility lemma at the selected gap endpoints and then prove that lemma.

\section{Preliminaries}\label{pre}

In this section, we introduce some basic notation and recall several auxiliary lemmas. For further details, see \cite{CAI2021109035,DamanikFillman22,DamanikFillman24,zhaozhiyan,You18}. We write $2\mathbb{T}^{d}=\mathbb{R}^{d}/(2\mathbb{Z})^{d}$ and use normalized Lebesgue measure on each torus.
For $h>0$, let
\begin{equation*}
\mathcal{S}_{h}{\coloneqq}\big\{\theta=(\theta_{1},\ldots,\theta_{d})\in\mathbb{C}^{d}:|\mathrm{Im}\,\theta_{i}|<h,\ i=1,\ldots,d\big\}.
\end{equation*}
For a scalar or matrix valued function $F$ on $\mathbb{T}^{d}$ admitting a holomorphic extension to $\mathcal{S}_{h}$, define
\begin{equation*}
|F|_{h}{\coloneqq}\sum_{n\in\mathbb{Z}^{d}}\|\widehat{F}(n)\|e^{{2\pi}h|n|}, \qquad \widehat{F}(n)\coloneqq\int_{[0,1]^{d}}F(\theta)e^{-2\pi\mathrm{i}\langle n,\theta\rangle}\,\mathrm{d}\theta.
\end{equation*}
where $|n|=|n_{1}|+\cdots+|n_{d}|.$ The corresponding spaces and Fourier norms on $2\mathbb{T}^{d}$ are defined analogously. Here $\|\cdot\|$ denotes the operator norm induced by the Euclidean norm for matrices and the absolute value for scalars. We denote by $C_{h}^{\omega}(\mathbb{T}^{d},\ast)$ the set of maps $F:\mathbb{T}^{d}\to\ast$ that admit a holomorphic extension to $\mathcal{S}_{h}$ and satisfy $|F|_{h}<\infty$, where $\ast$ usually denotes $\mathbb{R}$, $\mathrm{sl}(2,\mathbb{R})$, or $\mathrm{SL}(2,\mathbb{R})$. We also write
$C^{\omega}(\mathbb{T}^{d},\ast)=\bigcup_{h>0}C_{h}^{\omega}(\mathbb{T}^{d},\ast).$

For an integer $q\geqslant0$, we use the norm
\begin{equation*}
\|F\|_{q}\coloneqq\max_{|a|\leqslant q}\sup_{\theta\in\mathbb{T}^{d}}\|\partial^{a}F(\theta)\|,\qquad a\in\mathbb{Z}_{\geqslant0}^{d}.
\end{equation*}
For product estimates, we also use the equivalent submultiplicative norm
\begin{equation}\label{fn}
    \sum_{|a|\leqslant q}\frac{\|\partial^{a}F\|_{0}}{a!}, \qquad a!\coloneqq a_{1}!\cdots a_{d}!.
\end{equation}
Passing between these norms changes only constants depending on $q$ and $d.$

The same convention is used on $2\mathbb{T}^{d}$. For $N>0$, let
\begin{equation*}
\mathcal{T}_{N}F\coloneqq\sum_{|n|\leqslant N}\widehat{F}(n)e^{2\pi i\langle n,\theta\rangle},\qquad \mathcal{R}_{N}F\coloneqq F-\mathcal{T}_{N}F.
\end{equation*}

\subsection{\texorpdfstring{{Quasi-periodic cocycles and uniform hyperbolicity}}{Quasi-periodic cocycles and uniform hyperbolicity}}

Let $A\in C^{0}(\mathbb{T}^{d},\mathrm{SL}(2,\mathbb{C}))$. The quasi-periodic cocycle $(\alpha,A)$ is defined by
\begin{equation*}
\begin{aligned}
(\alpha,A):\mathbb{T}^{d}\times\mathbb{C}^{2}&\longrightarrow\mathbb{T}^{d}\times\mathbb{C}^{2},\\
(\theta,v)&\longmapsto(\theta+\alpha,A(\theta)v).
\end{aligned}
\end{equation*}
For $n\in\mathbb{Z}$, its iterates are given by $(\alpha,A)^{n}=(n\alpha,A_{n})$, where
\begin{equation*}
A_{n}(\theta)=
\begin{cases}
A(\theta+(n-1)\alpha)\cdots A(\theta+\alpha)A(\theta), & n\geqslant1,\\
I, & n=0,\\
A^{-1}(\theta+n\alpha)A^{-1}(\theta+(n+1)\alpha)\cdots A^{-1}(\theta-\alpha), & n<0.
\end{cases}
\end{equation*}

The cocycle $(\alpha,A)$ is called uniformly hyperbolic if there exist constants $C,c>0$ and a continuous splitting
$\mathbb{C}^{2}=E^{s}(\theta)\oplus E^{u}(\theta)$ into one-dimensional complex subspaces such that, for every $\theta\in\mathbb{T}^{d}$,
\begin{equation*}
\begin{aligned}
A(\theta)E^{s}(\theta)&=E^{s}(\theta+\alpha),\\
A(\theta)E^{u}(\theta)&=E^{u}(\theta+\alpha),
\end{aligned}
\end{equation*}
and for every integer $n\geqslant0$,
\begin{equation*}
\begin{aligned}
\|A_{n}(\theta)v\|&\leqslant Ce^{-cn}\|v\|, \qquad v\in E^{s}(\theta),\\
\|A_{-n}(\theta)v\|&\leqslant Ce^{-cn}\|v\|, \qquad v\in E^{u}(\theta).
\end{aligned}
\end{equation*}
An important class of $\mathrm{SL}(2,\mathbb{R})$ cocycles consists of the Schr\"odinger cocycles $(\alpha,S_{E}^{V})$. For $V\in C^{0}(\mathbb{T}^{d},\mathbb{R})$ and $E\in\mathbb{R}$, we define
\begin{equation*}
 S_{E}^{V}(\theta)=
\begin{pmatrix}
E-V(\theta) & -1\\
1 & 0
\end{pmatrix}.
\end{equation*}
These cocycles arise from the eigenvalue equation $H_{V,\alpha,\theta}u=Eu$. Indeed, any formal solution $u=(u_{n})_{n\in\mathbb{Z}}$ of $H_{V,\alpha,\theta}u=Eu$ satisfies
\begin{equation*}
\begin{pmatrix}
u_{n+1}\\
u_{n}
\end{pmatrix}
=S_{E}^{V}(\theta+n\alpha)
\begin{pmatrix}
u_{n}\\
u_{n-1}
\end{pmatrix},
\qquad n\in\mathbb{Z}.
\end{equation*}

A fundamental theorem of Johnson \cite{Johnson198654} states that, for every $E\in\mathbb{R}$,
\begin{equation*}
E\notin\Sigma_{V,\alpha}\quad\Longleftrightarrow\quad(\alpha,S_{E}^{V})\text{ is uniformly hyperbolic}.
\end{equation*}

\subsection{Conjugation and Reducibility}

Let $\ast\in\{\omega,k\}$ and $A_{1},A_{2}\in C^{\ast}(\mathbb{T}^{d},\allowbreak\mathrm{SL}(2,\mathbb{R}))$. We call $(\alpha,A_{1})$ and $(\alpha,A_{2})$ $C^{\ast}$ conjugate if there exists $Z\in C^{\ast}(2\mathbb{T}^{d},\mathrm{SL}(2,\mathbb{R}))$ such that
\begin{equation*}
    Z(\theta+\alpha)^{-1}A_{1}(\theta)Z(\theta)=A_{2}(\theta), \qquad \theta\in2\mathbb{T}^{d}.
\end{equation*}

Let $0<h'\leqslant h$ and $A\in C_{h}^{\omega}(\mathbb{T}^{d},\mathrm{SL}(2,\mathbb{R}))$. We call $(\alpha,A)$ $C_{h,h'}^{\omega}$ reducible if it admits a conjugacy $Z\in C_{h'}^{\omega}(2\mathbb{T}^{d},\mathrm{SL}(2,\mathbb{R}))$ to a constant cocycle.

Let $0\leqslant k'\leqslant k$ be integers and $A\in C^{k}(\mathbb{T}^{d},\mathrm{SL}(2,\mathbb{R}))$. We call $(\alpha,A)$ $C^{k,k'}$ almost reducible if the closure of its $C^{k'}$ conjugacy class in $C^{k'}(\mathbb{T}^{d},\mathrm{SL}(2,\mathbb{R}))$ contains a constant cocycle. We call it $C^{k,k'}$ reducible if its $C^{k'}$ conjugacy class contains a constant cocycle.

\subsection{Fibered rotation number}
Assume that $A\in C^{0}(\mathbb{T}^{d},\mathrm{SL}(2,\mathbb{R}))$ is homotopic to the identity. It induces the circle skew-product $F_{A}:\mathbb{T}^{d}\times\mathbb{S}^{1}\to\mathbb{T}^{d}\times\mathbb{S}^{1}$ defined by
\begin{equation*}
F_{A}(\theta,\omega)=\bigg(\theta+\alpha,\frac{A(\theta)\omega}{\|A(\theta)\omega\|}\bigg).
\end{equation*}
We identify $\mathbb{S}^{1}$ with ${\mathbb{R}/\mathbb{Z}}$ through the parametrization ${y\mapsto(\cos(2\pi y),\sin(2\pi y))}$. Since $F_{A}$ is homotopic to the identity, it admits a lift $\widetilde{F}_{A}:\mathbb{T}^{d}\times\mathbb{R}\to\mathbb{T}^{d}\times\mathbb{R}$ of the form
\begin{equation*}
\widetilde{F}_{A}(\theta,y)=(\theta+\alpha,y+\psi(\theta,y)),
\end{equation*}
where $\psi$ is continuous and $1$-periodic in $y$. Let $\mu$ be an $F_{A}$-invariant probability measure on $\mathbb{T}^{d}\times\mathbb{S}^{1}$ whose projection onto $\mathbb{T}^{d}$ is the normalized Lebesgue measure ${\mathrm{d}\theta}$. The fibered rotation number of $(\alpha,A)$ is defined by
\begin{equation*}
\rho(\alpha,A)=\int_{\mathbb{T}^{d}\times\mathbb{S}^{1}}\psi(\theta,y)\,\mathrm{d}\mu(\theta,y)\pmod{\mathbb{Z}}.
\end{equation*}
This defines $\rho(\alpha,A)\in\mathbb{R}/\mathbb{Z}$ independently of the lift and the invariant measure, see \cite{CAI2021109035,zhaozhiyan}.

For $\phi\in\mathbb{R}$, let
\begin{equation*}
R_{\phi}=
\begin{pmatrix}
\cos(2\pi\phi) & -\sin(2\pi\phi)\\
\sin(2\pi\phi) & \cos(2\pi\phi)
\end{pmatrix}.
\end{equation*}
For $B\in C^{0}(2\mathbb{T}^{d},\mathrm{SL}(2,\mathbb{R}))$, we write $\deg B=n\in\mathbb{Z}^{d}$ if $B$ is homotopic to $\theta\mapsto R_{\langle n,\theta\rangle/2}$. Suppose that $A_{1},A_{2}\in C^{0}(\mathbb{T}^{d},\mathrm{SL}(2,\mathbb{R}))$ are homotopic to the identity and satisfy
$B(\theta+\alpha)^{-1}A_{1}(\theta)B(\theta)=A_{2}(\theta)$
for some $B\in C^{0}(2\mathbb{T}^{d},\mathrm{SL}(2,\mathbb{R}))$. Then
\begin{equation*}
\rho(\alpha,A_{2})=\rho(\alpha,A_{1})-\frac{\langle\deg B,\alpha\rangle}{2}\pmod{\mathbb{Z}}.
\end{equation*}
In particular, the fibered rotation number is invariant under real conjugacies homotopic to the identity. Moreover, for two such maps $A$ and $B$, the degree satisfies
\begin{equation*}
 \deg(AB)=\deg A+\deg B.
\end{equation*}

For a cocycle $(\alpha,Ae^{f})$ homotopic to the identity, where $A\in\mathrm{SL}(2,\mathbb{R})$ is constant and $f\in C^{0}(\mathbb{T}^{d},\mathrm{sl}(2,\mathbb{R}))$ is sufficiently small, we have
\begin{equation}\label{rho_holder}
    \|\rho(\alpha,Ae^{f})-\rho(\alpha,A)\|_{\mathbb{R}/\mathbb{Z}}\leqslant\tilde{c}\|f\|_{0}^{\frac{1}{2}},
\end{equation}
where $\tilde{c}=\tilde{c}(\|A\|)$.

\subsection{Integrated Density of States and Spectral gap}
Let $\mu_{V,\alpha,\theta}^{\delta_{0}}$ be the spectral measure of $H_{V,\alpha,\theta}$ associated with the unit vector $\delta_{0}$. The integrated density of states (IDS) is defined by
\begin{equation*}
    N_{V,\alpha}(E)\coloneqq\int_{\mathbb{T}^{d}}\mu_{V,\alpha,\theta}^{\delta_{0}}\big((-\infty,E]\big)\,\mathrm{d}\theta,
\end{equation*}
where $\mathrm{d}\theta$ denotes the normalized Haar measure on $\mathbb{T}^{d}$, see \cite{AVRON1983369,DamanikF22}. Equivalently, the IDS is the limiting distribution of eigenvalues of finite restrictions. For $L\in\mathbb{N}$, let $H_{V,\alpha,\theta}^{L}$ be the restriction of $H_{V,\alpha,\theta}$ to $[-L,L]\cap\mathbb{Z}$ with Dirichlet boundary conditions $u_{-L-1}=u_{L+1}=0$, and set
\begin{equation*}
    N_{V,\alpha,\theta}^{L}(E)\coloneqq\frac{1}{2L+1}\#\big\{x\leqslant E:x\text{ is an eigenvalue of }H_{V,\alpha,\theta}^{L}\big\},
\end{equation*}
where eigenvalues are counted with multiplicity. Then $N_{V,\alpha}(E)=\lim_{L\to\infty}N_{V,\alpha,\theta}^{L}(E)$ for every $E\in\mathbb{R}$ and every $\theta\in\mathbb{T}^{d}$, see \cite{AVRON1983369}. The Schr\"odinger rotation number $\rho(\alpha,S_{E}^{V})\in[0,\frac{1}{2}]$ satisfies
$ N_{V,\alpha}(E)=1-2\rho(\alpha,S_{E}^{V}),$
see \cite{zhaozhiyan,You18}.

The IDS is constant on each spectral gap. By the gap-labelling theorem \cite{DamanikF22}, each bounded spectral gap has a unique label $m\in\mathbb{Z}^{d}\setminus\{0\}$ satisfying
\begin{equation*}
    2\rho(\alpha,S_{E}^{V})\equiv\langle m,\alpha\rangle\pmod{\mathbb{Z}}
\end{equation*}
for every $E$ in the gap. The IDS is continuous, and its associated measure has support $\Sigma_{V,\alpha}$, see \cite{AVRON1983369,DamanikF22}. 
Consequently, for every $m\in\mathbb{Z}^{d}\setminus\{0\}$, the level set
\begin{equation*}
  [E_{m}^{-},E_{m}^{+}]=\big\{E\in\mathbb{R}:N_{V,\alpha}(E)=1-\{\langle m,\alpha\rangle\}\big\}
\end{equation*}
is a nonempty compact interval with endpoints $E_{m}^{-},E_{m}^{+}\in\Sigma_{V,\alpha}$. Here $\{x\}\coloneqq x-\lfloor x\rfloor$ denotes the fractional part of $x$. The case $E_{m}^{-}=E_{m}^{+}$ corresponds to a collapsed gap.

\subsection{Analytic Approximation}

Let $k>\frac{d}{2}$ be an integer and $f\in C^{k}(\mathbb{T}^{d},\allowbreak\mathrm{sl}(2,\mathbb{R}))$. Using the analytic smoothing construction of Zehnder \cite{Zehnder75}, together with Fourier estimates, we can choose a sequence of trigonometric polynomials $f_{j}\in C_{{\frac{1}{j}}}^{\omega}(\mathbb{T}^{d},\mathrm{sl}(2,\mathbb{R}))$, $j\geqslant1$, and a constant $C'=C'(k,d)>0$ such that
\begin{equation}
\label{r2}
\begin{aligned}
\|f_{j}-f\|_{k}&\longrightarrow0 \quad \text{as }j\longrightarrow\infty,\\
|f_{j}|_{{\frac{1}{j}}}&\leqslant C'\|f\|_{k},\\
|f_{j+1}-f_{j}|_{{\frac{1}{j+1}}}&\leqslant C'j^{-k-1+\frac{d}{2}}\|f\|_{k}.
\end{aligned}
\end{equation}
The approximants are obtained by convolution with a fixed family of smoothing kernels whose Fourier transforms are smooth and compactly supported. These kernels are independent of the regularity index $k$. Consequently, if $f\in C^{\tilde{k}}(\mathbb{T}^{d},\mathrm{sl}(2,\mathbb{R}))$ for an integer $\tilde{k}\geqslant k$, the same sequence satisfies \eqref{r2} with $k$ replaced by $\tilde{k}$ and with a constant depending on $\tilde{k}$ and $d$.

\subsection{Auxiliary lemmas}

Recall that the Lie algebras $\mathrm{sl}(2,\mathbb{R})$ and $\mathrm{su}(1,1)$ are isomorphic under the map $C\mapsto MCM^{-1}$, where
\begin{equation*}
M=\frac{1}{1+\mathrm{i}}
\begin{pmatrix}
1 & -\mathrm{i}\\
1 & \mathrm{i}
\end{pmatrix}.
\end{equation*}
The matrix $M$ is unitary, so conjugation by $M$ preserves the Euclidean operator norm. With this choice,
\begin{equation*}
MR_{\phi}M^{-1}=\operatorname{diag}(e^{-{2\pi}\mathrm{i}\phi},e^{{2\pi}\mathrm{i}\phi}).
\end{equation*}
We use $\mathfrak{B}_{h}(\mathbb{T}^{d},\ast)$ for the class of $\ast$-valued functions with finite Fourier weighted norm $|\cdot|_{h}$. For $B\in\mathfrak{B}_{h}(\mathbb{T}^{d},\mathrm{SU}(1,1))$, define
\begin{equation*}
\operatorname{Ad}(B)W=BWB^{-1}, \qquad W\in\mathfrak{B}_{h}(\mathbb{T}^{d},\mathrm{su}(1,1)).
\end{equation*}
On the real torus, write
\begin{equation}
\label{W} W=
\begin{pmatrix}
\mathrm{i}u & w\\
\overline{w} & -\mathrm{i}u
\end{pmatrix},
\end{equation}
where $u$ is real-valued and $w$ is complex-valued. Since conjugation by $B$ preserves $\mathrm{su}(1,1)$, we can write
\begin{equation}
\label{r3} W_{+}=BWB^{-1}=
\begin{pmatrix}
\mathrm{i}u_{+} & w_{+}\\
\overline{w_{+}} & -\mathrm{i}u_{+}
\end{pmatrix}.
\end{equation}
Angle brackets denote the normalized torus average, so $\langle W\rangle=\widehat{W}(0)$, and $[C,W]=CW-WC$ denotes the commutator.

\begin{lemma}\label{lemma2.1}
Let $C,P\in\mathrm{SU}(1,1)$ with $PCP^{-1}=\operatorname{diag}(e^{2\pi i\rho},e^{-2\pi i\rho})$ and $\sin(2\pi\rho)\neq0$. Write $W$ and $PWP^{-1}$ as in \eqref{W} and \eqref{r3}. Then
\begin{equation*}
|\langle w_{+}\rangle|\geqslant\frac{\|[C,\langle W\rangle]\|}{2|\sin(2\pi\rho)|\|P\|^{2}}.
\end{equation*}
If $P_{11}=a$, then for every $n\in\mathbb{Z}^{d}$,
\begin{equation*}
\begin{aligned}
|\widehat{w}_{+}(n)|&\geqslant |a|^{2}\max\bigg\{|\widehat{w}(n)|-3\max\{|\widehat{w}(-n)|,|\widehat{u}(n)|\},0\bigg\},\\
|a|^{2}&=\frac{(\|P\|+\|P\|^{-1})^{2}}{4}.
\end{aligned}
\end{equation*}
\end{lemma}

\begin{proof}
Write $P=\begin{pmatrix}a&b\\\overline{b}&\overline{a}\end{pmatrix}$. Direct multiplication gives $w_{+}=a^{2}w-b^{2}\overline{w}-2iab u$. Since $|b|\leqslant|a|$, the second estimate follows from the triangle inequality and $|\widehat{w}_{+}(n)|\geqslant0$. Moreover,
\begin{equation*}
\|P[C,\langle W\rangle]P^{-1}\|=2|\sin(2\pi\rho)||\langle w_{+}\rangle|,
\end{equation*}
which proves the first estimate.
\end{proof}

\begin{lemma}\cite{HouY12}\label{lemma2.2}
Let $W \in \mathfrak{B}_{h}(\mathbb{T}^{d},\mathrm{su}(1,1))$ and $B \in \mathfrak{B}_{h}(\mathbb{T}^{d}, PSU(1, 1))$ satisfying $|B-I|_{h} \leqslant \frac{1}{2}$. Write $W_{+} = B^{-1}WB$ in the form \eqref{r3}. Then
\begin{equation*}
|\widehat{w}_{+}(k)| \geqslant |\widehat{w}(k)| - 4 |B-I|_{h} |W|_{h} e^{-{2\pi}|k|h}, \quad \text{for all } k \in \mathbb{Z}^{d}.
\end{equation*}
\end{lemma}

\section{Dynamical estimates of \texorpdfstring{$C^{k}$}{Ck} quasi-periodic cocycle}

We first establish quantitative almost reducibility for $C^{k}$ cocycles. For cocycles that are not uniformly hyperbolic and whose rotation number is rational with respect to $\alpha$, we then obtain a quantitative parabolic reduction. Finally, we apply these estimates to bound the spectral gaps. Throughout this section, let $\alpha\in\mathrm{DC}_{d}(\gamma,\tau)$ and $A\in\mathrm{SL}(2,\mathbb{R})$. Conjugacy identities in the normal-form arguments are understood in $\mathrm{PSL}(2,\mathbb{R})$. Fixed numerical factors are absorbed into the constants. We write $\|x\|_{\mathbb{R}/\mathbb{Z}}\coloneqq\operatorname{dist}(x,\mathbb{Z})$.

\subsection{Choice of parameters and Quantitative reducibility}

Fix $0<\sigma<\frac{1}{6}$ and an integer $k>\big(\frac{30}{\sigma}+2\big)\tau+2$. Choose $D=D(k,\tau,\sigma)$ so that
\begin{equation}\label{choice-D-k}
    \max\bigg\{\frac{30}{\sigma},\frac{k-3}{\tau}-2\bigg\}<D<\frac{k-2}{\tau}-2.
\end{equation}
This interval is nonempty, and $(D+2)\tau+2<k<(D+2)\tau+3$. The same $D$ is used in \hyperref[GYGJMGYQ]{Theorem~\ref*{GYGJMGYQ}}(1).

Set $s\coloneqq\frac{1}{6D\tau+3}$. Let $c>0$ and $\tilde{D}>0$ be the constants in \hyperref[akam]{Theorem~\ref*{akam}}. For sufficiently  large integer $l_{1},$  define
\begin{equation}\label{sc}
\begin{aligned}
    &l_{j+1}\coloneqq\lceil l_{j}^{1+s}\rceil, \qquad r_{j}\coloneqq\frac{1}{2\pi l_{j}},\\
    &\varepsilon_{l_{j}}\coloneqq\frac{c}{(2\|A\|)^{\tilde{D}}l_{j}^{D\tau+\frac{1}{2}}}, \qquad N_{l_{j}}\coloneqq\frac{2|\ln\varepsilon_{l_{j}}|}{r_{j}-r_{j+1}}.
\end{aligned}
\end{equation}
The analytic approximations from Section 2.5 satisfy $\|f_{j}-f\|_{k}\to0$ and
\begin{equation}\label{sm}
\begin{aligned}
    \sup_{\theta\in\mathcal{S}_{1/(2\pi j)}}\|f_{j}(\theta)\|&\leqslant C'\|f\|_{k},\\
    \sup_{\theta\in\mathcal{S}_{1/(2\pi(j+1))}}\|f_{j+1}(\theta)-f_{j}(\theta)\|&\leqslant C'j^{-k}\|f\|_{k}.
\end{aligned}
\end{equation}
Here $C'\geqslant1$ depends only on $k,d$ and includes the constant in the telescoping estimate.

We recall the analytic KAM step used below.
\begin{theorem}[{\cite[Theorem B.1]{cai2026arithmetic}}]\label{akam}
Let $\alpha\in\mathrm{DC}_{d}(\gamma,\tau)$, $0<\gamma<1$, $\tau>d$, $0<\sigma<\frac{1}{6}$, $D>\frac{2}{\sigma}$, and $0<r'<r<1$. Set $N\coloneqq\frac{2|\ln\varepsilon|}{r-r'}$.There are constants $c=c(\gamma,\tau,d,\sigma,D)>0$ and $\tilde{D}=\tilde{D}(D)>0$ such that, if $A\in\mathrm{SL}(2,\mathbb{R})$ and $f:\mathbb{T}^{d}\to\mathrm{sl}(2,\mathbb{R})$ is holomorphic on $\mathcal{S}_{r}$ and satisfies, for some $0<\varepsilon<1$,
\begin{equation}\label{ks}
    \sup_{\theta\in\mathcal{S}_{r}}\|f(\theta)\|\leqslant\varepsilon\leqslant\frac{c}{\|A\|^{\tilde{D}}}(r-r')^{D\tau},
\end{equation}
there exist an analytic map $B:2\mathbb{T}^{d}\to\mathrm{SL}(2,\mathbb{R})$, a constant $A_{+}\in\mathrm{SL}(2,\mathbb{R})$, and an analytic map $f_{+}:\mathbb{T}^{d}\to\mathrm{sl}(2,\mathbb{R})$, holomorphic on $\mathcal{S}_{r'}$, such that
\begin{equation}\label{kc}
    B(\theta+\alpha)^{-1}Ae^{f(\theta)}B(\theta)=A_{+}e^{f_{+}(\theta)}.
\end{equation}

\begin{enumerate}
\item[(i)] If $\|2\rho(\alpha,A)-\langle n,\alpha\rangle\|_{\mathbb{R}/\mathbb{Z}}\geqslant\varepsilon^{\sigma}$ for every $n\in\mathbb{Z}^{d}$ with $0<|n|\leqslant N$, then
\begin{equation}\label{kn}
\begin{aligned}
    \sup_{\theta\in\mathcal{S}_{r'}}\|B(\theta)-I\|&\leqslant\varepsilon^{1-2\sigma-\frac{3}{D}},\\
    \sup_{\theta\in\mathcal{S}_{r'}}\|f_{+}(\theta)\|&\leqslant\varepsilon^{2-2\sigma-\frac{3}{D}}, \qquad \|A_{+}-A\|\leqslant2\|A\|\varepsilon.
\end{aligned}
\end{equation}

\item[(ii)] If $\|2\rho(\alpha,A)-\langle n^{*},\alpha\rangle\|_{\mathbb{R}/\mathbb{Z}}<\varepsilon^{\sigma}$ for some $n^{*}\in\mathbb{Z}^{d}$ with $0<|n^{*}|\leqslant N$, then there exist $P\in\mathrm{SL}(2,\mathbb{R})$ and $Y:\mathbb{T}^{d}\to\mathrm{sl}(2,\mathbb{R})$, holomorphic on $\mathcal{S}_{r'}$, such that
\begin{equation}\label{kp}
\begin{aligned}
    &B(\theta)=P^{-1}e^{Y(\theta)}R_{\langle n^{*},\theta\rangle/2}, \qquad \|P\|\leqslant4\|A\|^{\frac{1}{2}}\gamma^{-\frac{1}{2}}|n^{*}|^{\frac{\tau}{2}},\\
    &\sup_{\theta\in\mathcal{S}_{r'}}\|Y(\theta)\|\leqslant4\|A\|^{\frac{1}{2}}\gamma^{-\frac{1}{2}}|n^{*}|^{\frac{\tau}{2}}\varepsilon^{\frac{1}{2}}.
\end{aligned}
\end{equation}
Moreover,
\begin{equation}\label{kr}
\begin{aligned}
    \sup_{\theta\in\mathcal{S}_{r'}}\|B(\theta)\|&\leqslant8\|A\|^{\frac{1}{2}}\gamma^{-\frac{1}{2}}|n^{*}|^{\frac{\tau}{2}}e^{\pi|n^{*}|r'},\\
    \|B\|_{0}&\leqslant8\|A\|^{\frac{1}{2}}\gamma^{-\frac{1}{2}}|n^{*}|^{\frac{\tau}{2}}, \qquad \sup_{\theta\in\mathcal{S}_{r'}}\|f_{+}(\theta)\|\leqslant\varepsilon^{100}.
\end{aligned}
\end{equation}
There exist $t\in\mathbb{R}$ and $\nu\in\mathbb{C}$ such that
\begin{equation}\label{kf}
    A_{+}=M^{-1}\exp\begin{pmatrix}it&\nu\\\overline{\nu}&-it\end{pmatrix}M,
\end{equation}
with
\begin{equation}\label{kt}
    |t|\leqslant4\varepsilon^{\sigma}, \qquad |\nu|\leqslant64\|A\|\gamma^{-1}|n^{*}|^{\tau}\varepsilon e^{-2\pi|n^{*}|r}.
\end{equation}
\end{enumerate}
\end{theorem}

 Choose $C_{0}\geqslant128$, with $C_{0}\geqslant16(1+\|A\|)$, so that the coefficients in \eqref{kp} and \eqref{kt} are bounded by $C_{0}$ whenever the input constant has norm at most $2\|A\|$. Fix $L\geqslant4$ with $L^{\tau/2}\geqslant C_{0}$. The choice \eqref{choice-D-k} implies
\begin{equation}\label{par}
\begin{aligned}
    &s\bigg(D\tau+\frac{1}{2}\bigg)=\frac{1}{6}, \qquad 1-2\sigma-\frac{3}{D}>\frac{1}{2}, \qquad 2-2\sigma-\frac{3}{D}>1+s,\\
    &k-1-2\tau-(1+s)\bigg(D\tau+\frac{1}{2}\bigg)>\frac{1}{3}, \qquad \bigg(D\tau+\frac{1}{2}\bigg)\sigma>2\tau.
\end{aligned}
\end{equation}
For sufficiently large  $l_{1},$  we have 
\begin{equation}\label{l0}
    l_{1}^{s}\geqslant2, \qquad l_{1}^{\frac{1}{2}}\geqslant(4\pi)^{D\tau}, \qquad C_{0}\varepsilon_{l_{1}}^{\sigma}\leqslant\frac{1}{16},
\end{equation}
and for every $j\geqslant1$,
\begin{equation}\label{l1}
\begin{aligned}
    &r_{j+1}N_{l_{j}}\leqslant\frac{1}{4\pi}, \qquad 2C_{0}\varepsilon_{l_{j}}^{\sigma}(LN_{l_{j}}^{2})^{\tau}\leqslant\gamma,\\
    &C_{0}N_{l_{j}}^{\tau}\varepsilon_{l_{j}}\leqslant\varepsilon_{l_{j}}^{\sigma}, \qquad C_{0}^{3}\bigg(\varepsilon_{l_{j}}^{2-2\sigma-\frac{3}{D}}+N_{l_{j}}^{2\tau}\varepsilon_{l_{1}}l_{j}^{-(k-1)}\bigg)\leqslant\varepsilon_{l_{j+1}},
\end{aligned}
\end{equation}
together with
\begin{equation}\label{ls}
    C_{0}\varepsilon_{l_{1}}^{\sigma}+2\sum_{j=1}^{\infty}\varepsilon_{l_{j}}<\ln\frac{3}{2}, \qquad C_{0}\sum_{j=1}^{\infty}\bigg(\varepsilon_{l_{j}}^{\frac{1}{2}}+N_{l_{j}}^{\frac{\tau}{2}}\varepsilon_{l_{j}}^{\frac{1}{2}}\bigg)\leqslant1.
\end{equation}
These conditions can be imposed simultaneously. Indeed, $N_{l_{j}}\leqslant C l_{j}|\ln\varepsilon_{l_{j}}|$ and $(D\tau+\frac{1}{2})(1-\sigma)>\tau$. Together with \eqref{par}, these inequalities make the required ratios tend to zero uniformly in $j$ as $l_{1}\to\infty$. Both series in \eqref{ls} tend to zero as well. Having fixed $l_{1}$, set
\begin{equation}\label{r4}
    \varepsilon\coloneqq\frac{\varepsilon_{l_{1}}}{C'}=\frac{c}{C'(2\|A\|)^{\tilde{D}}l_{1}^{D\tau+\frac{1}{2}}}.
\end{equation}

With these choices, successive applications of \hyperref[akam]{Theorem~\ref*{akam}} yield the following iteration lemma.

\begin{lemma}\label{GZQQZZ}
Let $f\in C^{k}(\mathbb{T}^{d},\mathrm{sl}(2,\mathbb{R}))$ satisfy $\|f\|_{k}\leqslant\varepsilon$. For every $j\geqslant1$, there exist analytic maps $B_{l_{j}}:2\mathbb{T}^{d}\to\mathrm{SL}(2,\mathbb{R})$ and $f'_{l_{j}}:\mathbb{T}^{d}\to\mathrm{sl}(2,\mathbb{R})$, holomorphic on $\mathcal{S}_{r_{j+1}}$, and constant matrices $A_{l_{j}}\in\mathrm{SL}(2,\mathbb{R})$ such that
\begin{equation}\label{cq}
    B_{l_{j}}(\theta+\alpha)^{-1}Ae^{f_{l_{j}}(\theta)}B_{l_{j}}(\theta)=A_{l_{j}}e^{f'_{l_{j}}(\theta)}.
\end{equation}
 and
\begin{equation}\label{gb}
\begin{aligned}
    \|B_{l_{j}}\|_{0}&\leqslant\sup_{\theta\in\mathcal{S}_{r_{j+1}}}\|B_{l_{j}}(\theta)\|\leqslant C N_{l_{j}}^{\tau},\\
    \sup_{\theta\in\mathcal{S}_{r_{j+1}}}\|f'_{l_{j}}(\theta)\|&\leqslant\varepsilon_{l_{j}}^{2-2\sigma-\frac{3}{D}}, \qquad \|A_{l_{j}}\|\leqslant2\|A\|.
\end{aligned}
\end{equation}
Moreover, there exists $\overline{f}_{l_{j}}\in C^{k}(\mathbb{T}^{d},\mathrm{sl}(2,\mathbb{R}))$ satisfying
\begin{equation}\label{cf}
    B_{l_{j}}(\theta+\alpha)^{-1}Ae^{f(\theta)}B_{l_{j}}(\theta)=A_{l_{j}}e^{\overline{f}_{l_{j}}(\theta)}.
\end{equation}
For every integer $0\leqslant\tilde{k}\leqslant\frac{k-4\tau-3}{2(1+s)}$,
\begin{equation}\label{ck}
\begin{aligned}
    \|B_{l_{j}}\|_{\tilde{k}}&\leqslant C l_{j+1}^{\tilde{k}}N_{l_{j}}^{\tau}, \qquad \|\overline{f}_{l_{j}}\|_{0}\leqslant\varepsilon_{l_{j+1}},\\
    \|\overline{f}_{l_{j}}\|_{\tilde{k}}&\leqslant C l_{j+1}^{\tilde{k}}\varepsilon_{l_{j+1}}\longrightarrow0.
\end{aligned}
\end{equation}
\end{lemma}
Here and below, $C=C(\gamma,\tau,d,k,\sigma,\|A\|)$ is independent of $j.$ 

\begin{remark}
Set $A_{l_{0}}\coloneqq A$, $B_{l_{0}}\coloneqq I$, and $\widetilde{B}_{l_{j-1}}\coloneqq B_{l_{j-1}}^{-1}B_{l_{j}}$. At step $j$, the incremental conjugacy is supplied by \hyperref[akam]{Theorem~\ref*{akam}} with input constant $A_{l_{j-1}}$, threshold $\varepsilon_{l_{j}}$, and widths $r_{j},r_{j+1}$. The step is resonant when alternative (ii) applies, and its resonant vector is denoted by $n_{l_{j}}^{*}$. At such a step, the local logarithm $A''_{l_{j}}\coloneqq\log A_{l_{j}}$ and the coefficient $\nu_{l_{j}}$ in the corresponding normal form satisfy
\begin{equation}\label{r5}
    \|A''_{l_{j}}\|\leqslant C\varepsilon_{l_{j}}^{\sigma}, \qquad |\nu_{l_{j}}|\leqslant C|n_{l_{j}}^{*}|^{\tau}\varepsilon_{l_{j}}.
\end{equation}
\end{remark}
\begin{proof}

We argue by induction. For $j=1,$ by \eqref{sm}, \eqref{l0}, and \eqref{r4}, we have
\begin{equation*}
    \sup_{\theta\in\mathcal{S}_{r_{1}}}\|f_{l_{1}}(\theta)\|\leqslant C'\|f\|_{k}\leqslant\varepsilon_{l_{1}}, \qquad r_{1}-r_{2}\geqslant\frac{1}{4\pi l_{1}}.
\end{equation*}
Moreover, the definition of $\varepsilon_{l_{1}}$ and the inequality $l_{1}^{\frac{1}{2}}\geqslant(4\pi)^{D\tau}$ give
\begin{equation*}
    \varepsilon_{l_{1}}\leqslant\frac{c}{(2\|A\|)^{\tilde{D}}}\bigg(\frac{1}{4\pi l_{1}}\bigg)^{D\tau}\leqslant\frac{c}{\|A\|^{\tilde{D}}}(r_{1}-r_{2})^{D\tau}.
\end{equation*}
Thus \eqref{ks} holds with $(f,\varepsilon,r,r')=(f_{l_{1}},\varepsilon_{l_{1}},r_{1},r_{2})$. Applying \hyperref[akam]{Theorem~\ref*{akam}}, we obtain $B_{l_{1}}$, $A_{l_{1}}$, and $f'_{l_{1}}$ satisfying \eqref{cq} for $j=1$. 

If the first step is nonresonant, \eqref{kn} and \eqref{par} give
\begin{equation*}
    \sup_{\theta\in\mathcal{S}_{r_{2}}}\|B_{l_{1}}(\theta)-I\|\leqslant\varepsilon_{l_{1}}^{\frac{1}{2}}, \qquad \|A_{l_{1}}-A\|\leqslant2\|A\|\varepsilon_{l_{1}}.
\end{equation*}
By \eqref{sc} and \eqref{l0}, we have
\begin{equation*}
    \sup_{\theta\in\mathcal{S}_{r_{2}}}\|B_{l_{1}}(\theta)\|\leqslant2\leqslant2C_{0}N_{l_{1}}^{\tau}, \qquad \|A_{l_{1}}\|\leqslant(1+2\varepsilon_{l_{1}})\|A\|<2\|A\|.
\end{equation*}
The map $B_{l_{1}}$ is $\mathbb{T}^{d}$-periodic and homotopic to $I$, so $\deg B_{l_{1}}=0$.

If the first step is resonant, by \eqref{kp}, there exists $B_{l_{1}}$  with $\deg B_{l_{1}}=n_{l_{1}}^{*}$. And by \eqref{kr}, \eqref{l1}, we have
\begin{equation*}
\sup_{\theta\in\mathcal{S}_{r_{2}}}\|B_{l_{1}}(\theta)\|\leqslant8\|A\|^{\frac{1}{2}}\gamma^{-\frac{1}{2}}N_{l_{1}}^{\frac{\tau}{2}}e^{\pi r_{2}N_{l_{1}}}\leqslant2C_{0}N_{l_{1}}^{\tau}.
\end{equation*}
And \eqref{kt} and \eqref{l1} yield
\begin{equation*}
    |t_{l_{1}}|\leqslant4\varepsilon_{l_{1}}^{\sigma}, \qquad |\nu_{l_{1}}|\leqslant C_{0}|n_{l_{1}}^{*}|^{\tau}\varepsilon_{l_{1}}\leqslant\varepsilon_{l_{1}}^{\sigma}.
\end{equation*}
Together with \eqref{l0} implies
\begin{equation*}
    \|A''_{l_{1}}\|\leqslant5\varepsilon_{l_{1}}^{\sigma}, \qquad \|A_{l_{1}}\|\leqslant e^{5\varepsilon_{l_{1}}^{\sigma}}<2\leqslant2\|A\|.
\end{equation*}
 This proves \eqref{r5} at the first step.
In both cases, \eqref{kn} and \eqref{kr} imply
\begin{equation*}
    \sup_{\theta\in\mathcal{S}_{r_{2}}}\|f'_{l_{1}}(\theta)\|\leqslant\varepsilon_{l_{1}}^{2-2\sigma-\frac{3}{D}},
\end{equation*}
since $0<\varepsilon_{l_{1}}<1$ and $100>2-2\sigma-\frac{3}{D}$.

This proves the analytic estimates at the first step. Suppose that they hold through step $j$. We first pass from $f_{l_{j}}$ to $f_{l_{j+1}}$ and then apply the next analytic step.

Assume that all conclusions hold in step $j$. By \eqref{sm} and \eqref{r4},
\begin{equation*}
    \sup_{\theta\in\mathcal{S}_{r_{j+1}}}\|f_{l_{j+1}}(\theta)-f_{l_{j}}(\theta)\|\leqslant C'\|f\|_{k}\sum_{q=l_{j}}^{l_{j+1}-1}q^{-k}\leqslant\varepsilon_{l_{1}}l_{j}^{-(k-1)}.
\end{equation*}
Both approximations have norm at most $\varepsilon_{l_{1}}$ on $\mathcal{S}_{r_{j+1}}$.  The inductive bounds give
\begin{equation}\label{nf}
\begin{aligned}
    &\sup_{\theta\in\mathcal{S}_{r_{j+1}}}\Big\|e^{f'_{l_{j}}(\theta)}B_{l_{j}}(\theta)^{-1}e^{-f_{l_{j}}(\theta)}e^{f_{l_{j+1}}(\theta)}B_{l_{j}}(\theta)-I\Big\|\\
    &\qquad\leqslant2\varepsilon_{l_{j}}^{2-2\sigma-\frac{3}{D}}+64C_{0}^{2}N_{l_{j}}^{2\tau}\varepsilon_{l_{1}}l_{j}^{-(k-1)}\leqslant\frac{1}{2}\varepsilon_{l_{j+1}},
\end{aligned}
\end{equation}
where the last inequality follows from \eqref{l1} and $C_{0}\geqslant128$. We may thus define
\begin{equation}\label{fl}
    \widetilde{f}_{l_{j}}(\theta)\coloneqq\log\Big(e^{f'_{l_{j}}(\theta)}B_{l_{j}}(\theta)^{-1}e^{-f_{l_{j}}(\theta)}e^{f_{l_{j+1}}(\theta)}B_{l_{j}}(\theta)\Big),
\end{equation}
 By \eqref{nf}, the inequality $\|\log(I+X)\|\leqslant2\|X\|$ for $\|X\|\leqslant\frac{1}{2}$, and the width estimate used in the initial step,
\begin{equation}\label{fj}
    \sup_{\theta\in\mathcal{S}_{r_{j+1}}}\|\widetilde{f}_{l_{j}}(\theta)\|\leqslant\varepsilon_{l_{j+1}}\leqslant\frac{c}{\|A_{l_{j}}\|^{\tilde{D}}}(r_{j+1}-r_{j+2})^{D\tau},
\end{equation}
where we used $\|A_{l_{j}}\|\leqslant2\|A\|$. Also, \eqref{cq} and \eqref{fl} imply
\begin{equation}\label{ij}
    B_{l_{j}}(\theta+\alpha)^{-1}Ae^{f_{l_{j+1}}(\theta)}B_{l_{j}}(\theta)=A_{l_{j}}e^{\widetilde{f}_{l_{j}}(\theta)}.
\end{equation}
By \eqref{fj}, \hyperref[akam]{Theorem~\ref*{akam}} applies to $(\alpha,A_{l_{j}}e^{\widetilde{f}_{l_{j}}})$ with threshold $\varepsilon_{l_{j+1}}$ and widths $r_{j+1},r_{j+2}$. It gives $\widetilde{B}_{l_{j}}$, $A_{l_{j+1}}$, and $f'_{l_{j+1}}$ satisfying
\begin{equation*}
    \widetilde{B}_{l_{j}}(\theta+\alpha)^{-1}A_{l_{j}}e^{\widetilde{f}_{l_{j}}(\theta)}\widetilde{B}_{l_{j}}(\theta)=A_{l_{j+1}}e^{f'_{l_{j+1}}(\theta)}.
\end{equation*}
Set $B_{l_{j+1}}\coloneqq B_{l_{j}}\widetilde{B}_{l_{j}}$. Combining this identity with \eqref{ij} proves \eqref{cq} at step $j+1$.

If step $j+1$ is nonresonant, \eqref{kn} and \eqref{par} give
\begin{equation*}
\begin{aligned}
    &\sup_{\theta\in\mathcal{S}_{r_{j+2}}}\|\widetilde{B}_{l_{j}}(\theta)-I\|\leqslant\varepsilon_{l_{j+1}}^{\frac{1}{2}}, \qquad \deg\widetilde{B}_{l_{j}}=0,\\
    &\|A_{l_{j+1}}-A_{l_{j}}\|\leqslant4\|A\|\varepsilon_{l_{j+1}}, \qquad \sup_{\theta\in\mathcal{S}_{r_{j+2}}}\|f'_{l_{j+1}}(\theta)\|\leqslant\varepsilon_{l_{j+1}}^{2-2\sigma-\frac{3}{D}}.
\end{aligned}
\end{equation*}
If step $j+1$ is resonant, \eqref{kp} gives
\begin{equation*}
    \widetilde{B}_{l_{j}}(\theta)=P_{l_{j+1}}^{-1}e^{Y_{l_{j+1}}(\theta)}R_{\langle n_{l_{j+1}}^{*},\theta\rangle/2}, \qquad \deg\widetilde{B}_{l_{j}}=n_{l_{j+1}}^{*},
\end{equation*}
with
\begin{equation*}
    \|P_{l_{j+1}}\|\leqslant C_{0}|n_{l_{j+1}}^{*}|^{\frac{\tau}{2}}, \qquad \sup_{\theta\in\mathcal{S}_{r_{j+2}}}\|Y_{l_{j+1}}(\theta)\|\leqslant C_{0}|n_{l_{j+1}}^{*}|^{\frac{\tau}{2}}\varepsilon_{l_{j+1}}^{\frac{1}{2}}.
\end{equation*}
Moreover, \eqref{kf}, \eqref{kt}, and \eqref{l1} imply
\begin{equation*}
    |t_{l_{j+1}}|\leqslant4\varepsilon_{l_{j+1}}^{\sigma}, \qquad |\nu_{l_{j+1}}|\leqslant C_{0}|n_{l_{j+1}}^{*}|^{\tau}\varepsilon_{l_{j+1}}\leqslant\varepsilon_{l_{j+1}}^{\sigma}.
\end{equation*}
As in the initial step, the local logarithm is well defined and
\begin{equation}\label{ri}
    \|A''_{l_{j+1}}\|\leqslant5\varepsilon_{l_{j+1}}^{\sigma}, \qquad \|A_{l_{j+1}}-I\|\leqslant6\varepsilon_{l_{j+1}}^{\sigma}, \qquad \sup_{\theta\in\mathcal{S}_{r_{j+2}}}\|f'_{l_{j+1}}(\theta)\|\leqslant\varepsilon_{l_{j+1}}^{100}.
\end{equation}
This proves \eqref{r5} at step $j+1$. 

 At every nonresonant step $q$, \eqref{kn} gives $\|A_{l_{q}}\|\leqslant\|A_{l_{q-1}}\|(1+2\varepsilon_{l_{q}})$. If no resonance has occurred through step $j+1$, then $\|A_{l_{j+1}}\|\leqslant\|A\|\exp(2\sum_{q=1}^{j+1}\varepsilon_{l_{q}})<\frac{3}{2}\|A\|$ by \eqref{ls}. Otherwise, let $p\leqslant j+1$ be the last resonant step. The resonant normal form and \eqref{ls} yield
\begin{equation*}
    \|A_{l_{j+1}}\|\leqslant\exp\bigg(5\varepsilon_{l_{p}}^{\sigma}+2\sum_{q=p+1}^{j+1}\varepsilon_{l_{q}}\bigg)<\frac{3}{2}\leqslant\frac{3}{2}\|A\|.
\end{equation*}
Thus $\|A_{l_{j+1}}\|\leqslant2\|A\|$.

 Let $p<q\leqslant j+1$ be consecutive resonant steps. The resonant estimates and the intervening nonresonant increments give
\begin{equation*}
    \|A_{l_{q-1}}-I\|\leqslant6\varepsilon_{l_{p}}^{\sigma}+4\|A\|\sum_{h=p+1}^{q-1}\varepsilon_{l_{h}}\leqslant C_{0}\varepsilon_{l_{p}}^{\sigma},
\end{equation*}
since $\sum_{h=p+1}^{\infty}\varepsilon_{l_{h}}\leqslant2\varepsilon_{l_{p+1}}\leqslant2\varepsilon_{l_{p}}^{\sigma}$. For a constant matrix near $I$, its eigenvalues give
\begin{equation*}
    \|2\rho(\alpha,A_{l_{q-1}})\|_{\mathbb{R}/\mathbb{Z}}\leqslant\frac{1}{\pi}\|\log A_{l_{q-1}}\|\leqslant\frac{2}{\pi}\|A_{l_{q-1}}-I\|\leqslant C_{0}\varepsilon_{l_{p}}^{\sigma}.
\end{equation*}
The resonance condition at step $q$ and the Diophantine condition imply $\gamma|n_{l_{q}}^{*}|^{-\tau}\leqslant2C_{0}\varepsilon_{l_{p}}^{\sigma}$. Hence \eqref{l1} gives
\begin{equation}\label{sep}
    |n_{l_{q}}^{*}|\geqslant L N_{l_{p}}^{2}\geqslant L|n_{l_{p}}^{*}|^{2}.
\end{equation}
Let $j_{1}<\cdots<j_{r}\leqslant j+1$ be all the resonant steps. If $r\geqslant1$, \eqref{sep} and $L^{\tau/2}\geqslant C_{0}$ imply
\begin{equation}\label{pr}
    \prod_{i=1}^{r}C_{0}|n_{l_{j_{i}}}^{*}|^{\frac{\tau}{2}}\leqslant C_{0}|n_{l_{j_{r}}}^{*}|^{\tau}, \qquad \sum_{i=1}^{r}|n_{l_{j_{i}}}^{*}|\leqslant\frac{4}{3}|n_{l_{j_{r}}}^{*}|.
\end{equation}
Indeed, the first bound follows by induction on $r$, and the second follows from $|n_{l_{j_{i+1}}}^{*}|\geqslant4|n_{l_{j_{i}}}^{*}|$. By \eqref{ls}, the product of the norms of all factors close to $I$ is at most
\begin{equation*}
    \exp\bigg(C_{0}\sum_{q=1}^{j+1}\bigg(\varepsilon_{l_{q}}^{\frac{1}{2}}+N_{l_{q}}^{\frac{\tau}{2}}\varepsilon_{l_{q}}^{\frac{1}{2}}\bigg)\bigg)\leqslant e.
\end{equation*}
Also, $\|P_{l_{q}}^{-1}\|=\|P_{l_{q}}\|$ and $\sup_{\theta\in\mathcal{S}_{r_{j+2}}}\|R_{\langle n,\theta\rangle/2}\|\leqslant e^{\pi r_{j+2}|n|}$. The cutoffs increase with the step index, since
\begin{equation*}
    N_{l_{q}}\leqslant8\pi l_{q}|\ln\varepsilon_{l_{q}}|<4\pi l_{q+1}|\ln\varepsilon_{l_{q+1}}|\leqslant N_{l_{q+1}}.
\end{equation*}
Therefore, \eqref{l1} and \eqref{pr} give
\begin{equation*}
\begin{aligned}
    \sup_{\theta\in\mathcal{S}_{r_{j+2}}}\|B_{l_{j+1}}(\theta)\|&\leqslant eC_{0}N_{l_{j+1}}^{\tau}\exp\bigg(\frac{4\pi}{3}r_{j+2}N_{l_{j+1}}\bigg)\\
    &\leqslant e^{\frac{4}{3}}C_{0}N_{l_{j+1}}^{\tau}<4C_{0}N_{l_{j+1}}^{\tau}.
\end{aligned}
\end{equation*}

If no resonant step has occurred, the bound for $B_{l_{j+1}}$ follows from \eqref{ls}. This proves \eqref{gb} at step $j+1$ and completes the analytic induction.

Define $\overline{f}_{l_{j}}$ by the right-hand side of \eqref{fl}, with $f_{l_{j+1}}$ replaced by $f$. By \eqref{sm} and \eqref{r4}, the product estimate in \eqref{nf} applies on the real torus. Thus the local logarithm is well defined, $\overline{f}_{l_{j}}\in C^{k}(\mathbb{T}^{d},\mathrm{sl}(2,\mathbb{R}))$, and $\|\overline{f}_{l_{j}}\|_{0}\leqslant\varepsilon_{l_{j+1}}$. Equation \eqref{cq} then gives \eqref{cf}.

For $0\leqslant p\leqslant\tilde{k}<k-1$, Cauchy's inequalities applied to the analytic differences in \eqref{sm}, followed by summation, give $\|f-f_{l_{j}}\|_{p}\leqslant C\varepsilon_{l_{1}}l_{j}^{p-(k-1)}$. Using this estimate and \eqref{gb} in the product and chain rules for the local logarithm proves the bound for $\|\overline{f}_{l_{j}}\|_{\tilde{k}}$ in \eqref{ck}, by \eqref{l1}. The bound for $\|B_{l_{j}}\|_{\tilde{k}}$ follows from \eqref{gb} and Cauchy's inequalities. Finally, the convergence in \eqref{ck} follows from \eqref{sc} and $\tilde{k}<D\tau+\frac{1}{2}$.

\end{proof}

We next consider cocycles whose rotation number is rational with respect to $\alpha$. We show that the resonant steps terminate, so the conjugacies constructed above converge and yield a quantitative parabolic reduction.
At the choice of the initial scale, before fixing $\varepsilon$ in \eqref{r4}, we additionally require $C\varepsilon_{l_{1}}^{\frac{1}{2}-\sigma}<\frac{1}{2}$ and
\begin{equation}\label{r0}
    C\bigg(\frac{c}{(2\|A\|)^{\tilde{D}}}\bigg)^{1-\frac{1}{3D}}(C')^{\frac{1}{3D}}l_{1}^{-\frac{D\tau-1}{6D}}\leqslant1.
\end{equation}

\begin{theorem}\label{r8}
Let $f\in C^{k}(\mathbb{T}^{d},\mathrm{sl}(2,\mathbb{R}))$ satisfy $\|f\|_{k}\leqslant\varepsilon$. Suppose that $(\alpha,Ae^{f})$ is not uniformly hyperbolic and
\begin{equation*}
    2\rho(\alpha,Ae^{f})\equiv\langle m,\alpha\rangle\pmod{\mathbb{Z}}, \qquad m\in\mathbb{Z}^{d}\setminus\{0\}.
\end{equation*}
For every integer $0\leqslant\tilde{k}\leqslant\frac{k-4\tau-3}{2(1+s)}$, there exist a projectively $\mathbb{T}^{d}$-periodic map $B_{1}\in C^{\tilde{k}}(2\mathbb{T}^{d},\mathrm{SL}(2,\mathbb{R}))$ and $\zeta\in\mathbb{R}$ such that
\begin{equation*}
    B_{1}(\theta+\alpha)^{-1}Ae^{f(\theta)}B_{1}(\theta)=\begin{pmatrix}1&\zeta\\0&1\end{pmatrix}\eqqcolon\tilde{H},
\end{equation*}
with
\begin{equation*}
    |\zeta|\leqslant\varepsilon^{\frac{1}{3D}}|m|^{-(k-4\tau-\frac{5}{2})}, \qquad \|B_{1}\|_{\tilde{k}}\leqslant C|m|^{2\tilde{k}+\tau}.
\end{equation*}
The constant $C$ is independent of $m$ and of the chosen initial scale.
\end{theorem}

\begin{proof}

Apply \hyperref[GZQQZZ]{Lemma~\ref*{GZQQZZ}}. By \eqref{cf} and the transformation law for the rotation number,
\begin{equation*}
    2\rho(\alpha,A_{l_{j}}e^{\overline{f}_{l_{j}}})\equiv\langle m-\deg B_{l_{j}},\alpha\rangle\pmod{\mathbb{Z}}.
\end{equation*}
Combining \eqref{rho_holder} with \eqref{ck}, we obtain
\begin{equation}\label{rr}
    \|2\rho(\alpha,A_{l_{j}})-\langle m-\deg B_{l_{j}},\alpha\rangle\|_{\mathbb{R}/\mathbb{Z}}\leqslant C\|\overline{f}_{l_{j}}\|_{0}^{\frac{1}{2}}\leqslant C\varepsilon_{l_{j+1}}^{\frac{1}{2}}.
\end{equation}
The constant is uniform in $j$, since $\|A_{l_{j}}\|\leqslant2\|A\|$. The initial scale conditions and \eqref{l1} imply
\begin{equation}\label{r9}
    C\varepsilon_{l_{j}}^{\frac{1}{2}}<\frac{1}{2}\varepsilon_{l_{j}}^{\sigma}, \qquad 2\varepsilon_{l_{j}}^{\sigma}<\frac{\gamma}{(2N_{l_{j}})^{\tau}}, \qquad j\geqslant1.
\end{equation}
The second inequality ensures uniqueness of the resonant vector by the Diophantine condition.

We first show that the iteration has a last resonant step. If $\deg B_{l_{j-1}}=m$ for some $j\geqslant2$, then \eqref{rr} and \eqref{r9} give, for every $0<|n|\leqslant N_{l_{j}}$,
\begin{equation*}
    \|2\rho(\alpha,A_{l_{j-1}})-\langle n,\alpha\rangle\|_{\mathbb{R}/\mathbb{Z}}\geqslant\frac{\gamma}{N_{l_{j}}^{\tau}}-C\varepsilon_{l_{j}}^{\frac{1}{2}}>\varepsilon_{l_{j}}^{\sigma}.
\end{equation*}
Thus step $j$ is nonresonant, and its incremental degree is zero. Repeating this argument shows that all subsequent steps are nonresonant and their accumulated degree remains $m$.

Since $N_{l_{j+1}}/N_{l_{j}}\to\infty$, choose $J\geqslant2$ such that $|m|+\sum_{j=1}^{J-1}N_{l_{j}}<N_{l_{J}}$. If $\deg B_{l_{J-1}}\neq m$, set $m'\coloneqq m-\deg B_{l_{J-1}}$. Degree additivity gives $0<|m'|<N_{l_{J}}$, while \eqref{rr} and \eqref{r9} imply
\begin{equation*}
    \|2\rho(\alpha,A_{l_{J-1}})-\langle m',\alpha\rangle\|_{\mathbb{R}/\mathbb{Z}}\leqslant C\varepsilon_{l_{J}}^{\frac{1}{2}}<\varepsilon_{l_{J}}^{\sigma}.
\end{equation*}
Consequently, step $J$ is resonant with $n_{l_{J}}^{*}=m'$, and $\deg B_{l_{J}}=m$. In either case, the accumulated degree reaches $m$ and no further resonance occurs. Since $m\neq0$, there is at least one resonant step. Denote all resonant steps by $j_{1}<\cdots<j_{r}=g$. Then
\begin{equation}\label{r10}
    m=\sum_{i=1}^{r}n_{l_{j_{i}}}^{*}, \qquad \deg B_{l_{j}}=m, \qquad j\geqslant g.
\end{equation}
By \eqref{sep}, $\sum_{i=1}^{r-1}|n_{l_{j_{i}}}^{*}|\leqslant\frac{1}{3}|n_{l_{g}}^{*}|$. Hence
\begin{equation}\label{m}
    \frac{2}{3}|n_{l_{g}}^{*}|\leqslant|m|\leqslant\frac{4}{3}|n_{l_{g}}^{*}|.
\end{equation}

 The parameter choices imply
\begin{equation*}
    (1+s)\tilde{k}-\frac{D\tau+\frac{1}{2}}{2}<-\tau-\frac{1}{4}, \qquad (1+s)\tilde{k}+\frac{\tau}{2}-\frac{D\tau+\frac{1}{2}}{2}<-\frac{\tau}{2}-\frac{1}{4}.
\end{equation*}
Since every step after $g$ is nonresonant, Cauchy's inequalities give
\begin{equation*}
    \|\widetilde{B}_{l_{j}}-I\|_{\tilde{k}}\leqslant C l_{j+2}^{\tilde{k}}\varepsilon_{l_{j+1}}^{\frac{1}{2}}\leqslant C l_{j+1}^{-\frac{1}{4}}, \qquad j\geqslant g.
\end{equation*}
It shows that $B_{l_{j}}$ converges in $C^{\tilde{k}}$ to a projectively $\mathbb{T}^{d}$-periodic map $B:2\mathbb{T}^{d}\to\mathrm{SL}(2,\mathbb{R})$. The resonant normal form and \eqref{ls} also give $\|A_{l_{j}}\|\leqslant\exp(5\varepsilon_{l_{g}}^{\sigma}+2\sum_{q=g+1}^{j}\varepsilon_{l_{q}})<\frac{3}{2}$ for $j\geqslant g$. Thus the constants converge, with
\begin{equation}\label{rg}
    \overline{A}\coloneqq\lim_{j\to\infty}A_{l_{j}}, \qquad \|\overline{A}-A_{l_{g}}\|\leqslant3\sum_{j=g+1}^{\infty}\varepsilon_{l_{j}}\leqslant6\varepsilon_{l_{g+1}}.
\end{equation}
Since $\|\overline{f}_{l_{j}}\|_{\tilde{k}}\to0$ by \eqref{ck}, passing to the limit in \eqref{cf} gives
\begin{equation*}
    B(\theta+\alpha)^{-1}Ae^{f(\theta)}B(\theta)=\overline{A}, \qquad \deg B=m, \qquad 2\rho(\alpha,\overline{A})\equiv0\pmod{\mathbb{Z}}.
\end{equation*}
Here the degree is preserved by uniform convergence. The matrix $\overline{A}$ is not hyperbolic, since the original cocycle is not uniformly hyperbolic. Moreover, \eqref{rg} and the resonant estimates place $\overline{A}$ near $I$. Its rotation number therefore implies $\operatorname{tr}\overline{A}=2$. Consequently, $h\coloneqq\log\overline{A}=\overline{A}-I$ satisfies $h^{2}=0$. Write
\begin{equation*}
    MhM^{-1}=\begin{pmatrix}it&\nu\\\overline{\nu}&-it\end{pmatrix}, \qquad t\in\mathbb{R}, \quad \nu\in\mathbb{C}, \quad |t|=|\nu|.
\end{equation*}
A constant rotation $R_{\phi}$ puts $h$ into the form
\begin{equation*}
    R_{\phi}hR_{-\phi}=\begin{pmatrix}0&\zeta\\0&0\end{pmatrix}, \qquad |\zeta|=\|h\|\leqslant2|\nu|.
\end{equation*}
Set $B_{1}\coloneqq B R_{-\phi}$. This gives the required conjugacy, with $\|B_{1}\|_{\tilde{k}}=\|B\|_{\tilde{k}}$.

We now estimate $\zeta$. The local logarithm is Lipschitz near $I$, so \eqref{r5} and \eqref{rg} imply
\begin{equation*}
    |\nu-\nu_{l_{g}}|\leqslant C\varepsilon_{l_{g+1}}, \qquad |\nu|\leqslant C|n_{l_{g}}^{*}|^{\tau}\varepsilon_{l_{g}}.
\end{equation*}
By \eqref{m} and $|n_{l_{g}}^{*}|\leqslant N_{l_{g}}\leqslant C l_{g}|\ln\varepsilon_{l_{g}}|$,
\begin{equation*}
\begin{aligned}
    |\nu||m|^{k-4\tau-\frac{5}{2}}&\leqslant\frac{C c}{(2\|A\|)^{\tilde{D}}}l_{g}^{k-(D+3)\tau-3}|\ln\varepsilon_{l_{g}}|^{k-3\tau-\frac{5}{2}}\\
    &\leqslant\frac{C c}{(2\|A\|)^{\tilde{D}}}l_{1}^{-\frac{\tau}{2}}.
\end{aligned}
\end{equation*}
The last bound follows from $k-(D+3)\tau-3<-\tau$ and $|\ln\varepsilon_{l_{j}}|\leqslant C(1+\ln l_{j})$. All constants are independent of $g$ and $l_{1}$. By \eqref{r4} and \eqref{r0},
\begin{equation*}
    \frac{2|\nu||m|^{k-4\tau-\frac{5}{2}}}{\varepsilon^{\frac{1}{3D}}}\leqslant C\bigg(\frac{c}{(2\|A\|)^{\tilde{D}}}\bigg)^{1-\frac{1}{3D}}(C')^{\frac{1}{3D}}l_{1}^{-\frac{D\tau-1}{6D}}\leqslant1.
\end{equation*}
Together with $|\zeta|\leqslant2|\nu|$, this proves the claimed estimate for $\zeta$.

Finally, we estimate the conjugacy. At every resonant step $j_{i}$, Cauchy's inequalities give
\begin{equation*}
    \|Y_{l_{j_{i}}}\|_{\tilde{k}}\leqslant C l_{j_{i}+1}^{\tilde{k}}N_{l_{j_{i}}}^{\frac{\tau}{2}}\varepsilon_{l_{j_{i}}}^{\frac{1}{2}}\leqslant C l_{j_{i}}^{-\frac{1}{4}}.
\end{equation*}
The same bound holds for $\|\widetilde{B}_{l_{q-1}}-I\|_{\tilde{k}}$ at a nonresonant step $q$. Hence all factors close to $I$ have a uniformly bounded product of norms. For the following product estimates, we use the norm in \eqref{fn} with $q=\tilde{k}$,
\begin{equation*}
    \|R_{\langle n,\theta\rangle/2}\|_{\tilde{k}}=\sum_{p=0}^{\tilde{k}}\frac{(\pi|n|)^{p}}{p!}\leqslant e^{\pi}|n|^{\tilde{k}}, \qquad |n|\geqslant1.
\end{equation*}
Combining this with $\|P_{l_{j_{i}}}^{-1}\|\leqslant4(2\|A\|)^{\frac{1}{2}}\gamma^{-\frac{1}{2}}|n_{l_{j_{i}}}^{*}|^{\frac{\tau}{2}}$, the choice of $C_{0}$ bounds the coefficient of each resonant factor by $2C_{0}$. Since $L^{\tilde{k}+\tau/2}\geqslant4C_{0}\geqslant2C_{0}$, \eqref{sep} implies, by induction on the number of resonances,
\begin{equation*}
    \|B_{1}\|_{\tilde{k}}\leqslant C\prod_{i=1}^{r}2C_{0}|n_{l_{j_{i}}}^{*}|^{\tilde{k}+\frac{\tau}{2}}\leqslant C|n_{l_{g}}^{*}|^{2\tilde{k}+\tau}\leqslant C|m|^{2\tilde{k}+\tau}.
\end{equation*}
Returning to the original $C^{\tilde{k}}$ norm changes only the constant $C$.

This completes the proof.
\end{proof}

\subsection{An upper bound for spectral gaps}

In this subsection, we prove part (1) of \hyperref[GYGJMGYQ]{Theorem~\ref*{GYGJMGYQ}} by combining \hyperref[r8]{Theorem~\ref*{r8}} with the upper bound for spectral gaps obtained via the Moser-P\"oschel argument.

 For each integer $q>3\tau+d$, choose
\begin{equation}\label{dt}
    D_{\tau}\coloneqq C(q,d,\tau)\bigg(1+\sum_{n\in\mathbb{Z}^{d}\setminus\{0\}}|n|^{-(q-3\tau)}\bigg), \qquad K\coloneqq D_{\tau}\gamma^{-3}\geqslant1,
\end{equation}
where the fixed constant $C(q,d,\tau)\geqslant1$ includes the constants in \cite[Lemma 4.1]{CAI2021109035}. The series converges since $q>3\tau+d$.

We recall the following upper bound for spectral gaps obtained by the Moser-P\"oschel argument, see \cite[Section 4.2]{CAI2021109035} and \cite[Sections 3.3 and 3.4]{HeCheng2025}.

\begin{theorem}[{\cite[Section 4.2]{CAI2021109035}}]\label{cw}
Let $\alpha\in\mathrm{DC}_{d}(\gamma,\tau)$ with $0<\gamma<1$ and $\tau>d$, and let $q>3\tau+d$ be an integer. Suppose that $V\in C^{q}(\mathbb{T}^{d},\mathbb{R})$ and a projectively $\mathbb{T}^{d}$-periodic map $B\in C^{q}(2\mathbb{T}^{d},\mathrm{SL}(2,\mathbb{R}))$ satisfies
\begin{equation}\label{constant}
    B(\theta+\alpha)^{-1}S_{E_{m}^{+}}^{V}(\theta)B(\theta)=\begin{pmatrix}1&\zeta\\0&1\end{pmatrix}, \qquad m\in\mathbb{Z}^{d}\setminus\{0\}, \qquad 0<\zeta<\frac{1}{2}.
\end{equation}
There exists $c_{0}=c_{0}(q,d,\tau)>0$ such that, if
\begin{equation}\label{cwh}
    0<t<\frac{1}{6}, \qquad K^{2}\|B\|_{q}^{6}\zeta^{3t}\leqslant c_{0},
\end{equation}
then $|G_{m}(V)|\leqslant\zeta^{1-t}$.
\end{theorem}

\begin{remark}
The argument in \cite[Section 4.2]{CAI2021109035} uses $t=\frac{1}{18}$, but the same proof applies to every $0<t<\frac{1}{6}$ by taking the energy displacement $\delta=\zeta^{1-t}$ and applying \cite[Lemma 4.2]{CAI2021109035} with $\kappa=t$. Under \eqref{cwh}, the error estimates are uniform in $t$, so $c_{0}$ can be chosen independently of $t$.
\end{remark}

\begin{proof}[{Proof of part (1) of {\hyperref[GYGJMGYQ]{Theorem~\ref*{GYGJMGYQ}}}}]
 Set
\begin{equation}\label{qt}
    q\coloneqq\lfloor3\tau+d+1\rfloor+1, \qquad t\coloneqq\frac{18\tau+8}{k-7\tau}.
\end{equation}
Since $\tau>d\geqslant1$, we have $3\tau+d<q<4\tau+2$. Moreover, $k>182\tau+2$ and $s<\frac{1}{3}$ give
\begin{equation*}
    \frac{k-4\tau-3}{2(1+s)}>\frac{3(178\tau-1)}{8}>4\tau+2>q, \qquad k-7\tau>175\tau+2>108\tau+48.
\end{equation*}
Thus $q$ is an admissible regularity in \hyperref[r8]{Theorem~\ref*{r8}} and $0<t<\frac{1}{6}$.

We apply the quantitative reducibility theorem to the Schr\"odinger cocycle $(\alpha,S_{E}^{V})$. It can be written as $(\alpha,A_{E}e^{f_{E}})$, where
\begin{equation*}
A_{E}=\begin{pmatrix}E&-1\\1&0\end{pmatrix}, \qquad f_{E}(\theta)=\begin{pmatrix}0&0\\V(\theta)&0\end{pmatrix}.
\end{equation*}
In particular, $\|f_{E}\|_{k}=\|V\|_{k}$. Since $|E|\leqslant2+\|V\|_{0}$ for $E\in\Sigma_{V,\alpha}$, we have $\|A_{E}\|\leqslant3$ when $\|V\|_{0}$ is sufficiently small. The threshold $\varepsilon$ in \hyperref[r8]{Theorem~\ref*{r8}} is therefore independent of the energy and the gap label. Consider a noncollapsed gap $G_{m}(V)$.
At $E=E_{m}^{+}$, the Schr\"odinger cocycle is not uniformly hyperbolic and its rotation number satisfies the resonance hypothesis of \hyperref[r8]{Theorem~\ref*{r8}}. 
There exists  $B\in C^{\tilde{k}}(2\mathbb{T}^{d},\mathrm{SL}(2,\mathbb{R}))$  such that
\begin{equation}
B(\theta+\alpha)^{-1}S_{E_{m}^{+}}^{V}(\theta)B(\theta)= H, \qquad H=\begin{pmatrix}1&\zeta\\0&1\end{pmatrix},
\end{equation}
with
\begin{equation}
\label{edge}
0<\zeta\leqslant\varepsilon^{\frac{1}{3D}}|m|^{-(k-4\tau-\frac{5}{2})}<\frac{1}{2}, \qquad \|B\|_{\tilde{k}}\leqslant C|m|^{2\tilde{k}+\tau}.
\end{equation}
Here $\tilde{k}\leqslant\frac{k-4\tau-3}{2(1+s)}$ and  $\zeta>0$ at the right endpoint of a noncollapsed gap, see \cite[Remark 4.1]{CAI2021109035}.

It remains to verify the smallness condition in \eqref{cwh}. Since $q<4\tau+2$ and $k-4\tau-\frac{5}{2}>k-7\tau$, \eqref{qt} gives
\begin{equation*}
    12q+6\tau<54\tau+24=3t(k-7\tau)<3t\bigg(k-4\tau-\frac{5}{2}\bigg).
\end{equation*}
Therefore, \eqref{edge}  imply
\begin{equation}\label{tildeP}
\begin{aligned}
    K^{2}\|B\|_{q}^{6}\zeta^{3t}
    &\leqslant C^{6}K^{2}\varepsilon^{\frac{t}{D}}|m|^{12q+6\tau-3t(k-4\tau-\frac{5}{2})}\\
    &\leqslant C^{6}K^{2}\varepsilon^{\frac{t}{D}}\leqslant c_{0}.
\end{aligned}
\end{equation}
 Consequently, by \hyperref[cw]{Theorem~\ref*{cw}},
\begin{equation*}
    |G_{m}(V)|\leqslant\zeta^{1-t}\leqslant\varepsilon^{\frac{1-t}{3D}}|m|^{-(1-t)(k-4\tau-\frac{5}{2})}.
\end{equation*}
Since $D>180$ and $t<\frac{1}{6}$, we have $\frac{1-t}{3D}>\frac{5}{18D}>\frac{1}{6(D-5)}$. Furthermore,
\begin{equation*}
\begin{aligned}
    (1-t)\bigg(k-4\tau-\frac{5}{2}\bigg)
    &=k-22\tau-\frac{21}{2}-t\bigg(3\tau-\frac{5}{2}\bigg)\\
    &>k-\frac{45}{2}\tau-\frac{21}{2}>k-34\tau.
\end{aligned}
\end{equation*}
As $0<\varepsilon<1$ and $|m|\geqslant1$, it follows that
\begin{equation*}
    |G_{m}(V)|\leqslant\varepsilon^{\frac{1-t}{3D}}|m|^{-(1-t)(k-4\tau-\frac{5}{2})}\leqslant\varepsilon^{\frac{1}{6(D-5)}}|m|^{-(k-34\tau)}.
\end{equation*}

\end{proof}

\subsection{A lower bound for spectral gaps}

We prove a lower bound following the Moser-P\"oschel argument in \cite[Section 6]{HeCheng2025,zhaozhiyan}. This estimate will be used to prove part (2) of \hyperref[GYGJMGYQ]{Theorem~\ref*{GYGJMGYQ}}. 

\begin{lemma}\label{r30}
Let $q>3\tau+d$ be an integer and let $K$ be as in \eqref{dt}.
Suppose that $E\in\Sigma_{V,\alpha},$ $0<N_{V,\alpha}(E)<1$ and $B\in C^{q}(2\mathbb{T}^{d},\mathrm{SL}(2,\mathbb{R}))$ satisfy
\begin{equation}\label{lp}
    B(\theta+\alpha)^{-1}S_{E}^{V}(\theta)B(\theta)=H,\qquad H=\begin{pmatrix}1&\zeta\\0&1\end{pmatrix},\qquad 0<\zeta<\frac{1}{2}.
\end{equation}
There are constants $c_{0},c_{1}>0$, depending only on $q,d,\tau$, such that
\begin{equation}\label{r32}
    K^{2}\zeta\|B\|_{q}^{4}\leqslant c_{0}
\end{equation}
implies
\begin{equation}\label{r33}
    (E-c_{1}\zeta\|B\|_{0}^{-2},E)\cap\Sigma_{V,\alpha}=\varnothing.
\end{equation}
In particular, $E$ is the right endpoint of a noncollapsed spectral gap. If its label is $m$, then
\begin{equation}\label{r34}
    |G_{m}(V)|\geqslant c_{1}\zeta\|B\|_{0}^{-2}.
\end{equation}
\end{lemma}

\begin{proof}
Write $u\coloneqq B_{11}$, $v\coloneqq B_{12}$, and set $a\coloneqq[u^{2}]$, $b\coloneqq[uv]$, $c\coloneqq[v^{2}]$ and $\Delta\coloneqq ac-b^{2}$. By \eqref{lp} and $\det B=1$, direct calculation gives
\begin{equation}\label{pe}
    B(\theta+\alpha)^{-1}S_{E-s}^{V}(\theta)B(\theta)=H-sP(\theta),\qquad P(\theta)\coloneqq\begin{pmatrix}uv-\zeta u^{2}&v^{2}-\zeta uv\\-u^{2}&-uv\end{pmatrix},
\end{equation}
with
\begin{equation}\label{pm}
    \|P\|_{q}\leqslant C\|B\|_{q}^{2}.
\end{equation}

Choose $c_{1}\in(0,\frac{1}{2})$ sufficiently small. For $0<s\leqslant s_{*}\coloneqq c_{1}\zeta\|B\|_{0}^{-2}$, condition \eqref{r32} and $\|B\|_{q}\geqslant\|B\|_{0}\geqslant1$ imply
\begin{equation}\label{ss}
    Ks\|B\|_{q}^{2}\leqslant c_{1}K\zeta\|B\|_{q}^{2}\leqslant c_{1}K^{2}\zeta\|B\|_{q}^{4}\leqslant c_{1}c_{0}<\frac{1}{100}.
\end{equation}
By \eqref{pe}, \eqref{pm} and \eqref{ss}, we  apply \cite[Lemma 4.1]{CAI2021109035} with $\tilde{k}=q$ and $\hat{k}=0$. This is the averaging lemma recalled in \cite[Lemma 6.1]{HeCheng2025}. There exist $\widetilde{B}_{s}\in C^{0}(2\mathbb{T}^{d},\mathrm{SL}(2,\mathbb{R}))$ and $P_{1,s}\in C^{0}(\mathbb{T}^{d},\mathrm{gl}(2,\mathbb{R}))$ such that
\begin{equation}\label{av}
    \widetilde{B}_{s}(\theta+\alpha)^{-1}(H-sP(\theta))\widetilde{B}_{s}(\theta)=e^{b_{0}-sb_{1}}+s^{2}P_{1,s}(\theta),
\end{equation}
where
\begin{equation}\label{bp}
    b_{0}\coloneqq\begin{pmatrix}0&\zeta\\0&0\end{pmatrix},\qquad b_{1}\coloneqq\begin{pmatrix}b-\frac{\zeta}{2}a&c-\zeta b\\-a&-b+\frac{\zeta}{2}a\end{pmatrix},
\end{equation}
and
\begin{equation}\label{ar}
    \|P_{1,s}\|_{0}\leqslant CK^{2}\|B\|_{q}^{4}+Cs^{-1}\zeta^{2}\|B\|_{q}^{2}.
\end{equation}

We next estimate the determinant of $b_{0}-sb_{1}$. The second row of \eqref{lp} gives $B_{21}(\theta+\alpha)=u(\theta)$ and $B_{22}(\theta+\alpha)=v(\theta)-\zeta u(\theta)$. Using $\det B=1$, integration and the Cauchy-Schwarz inequality yield
\begin{equation*}
    1\leqslant2\sqrt{ac}+\zeta a\leqslant3\|B\|_{0}\sqrt{a}.
\end{equation*}
Hence
\begin{equation}\label{r35}
    a\geqslant\frac{1}{9\|B\|_{0}^{2}},\qquad 0\leqslant\Delta\leqslant ac\leqslant a\|B\|_{0}^{2}.
\end{equation}
By \eqref{bp},
\begin{equation}\label{dm}
    d(s)\coloneqq\det(b_{0}-sb_{1})+\frac{1}{4}s^{2}\zeta^{2}a^{2}=-s\zeta a+s^{2}\Delta.
\end{equation}
Since $s\|B\|_{0}^{2}\leqslant c_{1}\zeta$ and $c_{1}<\frac{1}{2}$, \eqref{r35} and \eqref{dm} imply
\begin{equation}\label{dh}
    \det(b_{0}-sb_{1})\leqslant d(s)\leqslant-\frac{1}{2}s\zeta a<0,\qquad \lambda_{s}\coloneqq\sqrt{-\det(b_{0}-sb_{1})}\geqslant\sqrt{\frac{s\zeta a}{2}}.
\end{equation}
Moreover, \eqref{bp} and $s\|B\|_{0}^{2}\leqslant c_{1}\zeta$ give
\begin{equation}\label{dn}
    \|b_{0}-sb_{1}\|\leqslant C(\zeta+s\|B\|_{0}^{2})\leqslant C\zeta.
\end{equation}
By \eqref{dh}, the matrix $b_{0}-sb_{1}$ has eigenvalues $\lambda_{s}$ and $-\lambda_{s}$. By \cite[proof of Theorem 6.2]{zhaozhiyan} and \cite[proof of Proposition 18]{Puig06}, there exists $L_{s}\in\mathrm{SL}(2,\mathbb{R})$ such that
\begin{equation}\label{dl}
    L_{s}e^{b_{0}-sb_{1}}L_{s}^{-1}=\begin{pmatrix}e^{\lambda_{s}}&0\\0&e^{-\lambda_{s}}\end{pmatrix},\qquad \|L_{s}\|^{2}\leqslant C\frac{\|b_{0}-sb_{1}\|}{\lambda_{s}}\leqslant C\frac{\zeta}{\lambda_{s}},
\end{equation}
where the last inequality follows from \eqref{dn}.
Set $F_{s}\coloneqq s^{2}L_{s}P_{1,s}L_{s}^{-1}$. Using \eqref{ar}, \eqref{dh}, \eqref{dl} and $\|L_{s}^{-1}\|=\|L_{s}\|$, we obtain
\begin{equation}\label{fr}
\begin{aligned}
    \frac{\|F_{s}\|_{0}}{\lambda_{s}}
    &\leqslant C\frac{\zeta}{\lambda_{s}^{2}}\bigg(K^{2}s^{2}\|B\|_{q}^{4}+s\zeta^{2}\|B\|_{q}^{2}\bigg)\\
    &\leqslant C\bigg(\frac{K^{2}s\|B\|_{q}^{4}}{a}+\frac{\zeta^{2}\|B\|_{q}^{2}}{a}\bigg)\\
    &\leqslant C\bigg(K^{2}\zeta\|B\|_{q}^{4}+\zeta^{2}\|B\|_{q}^{2}\|B\|_{0}^{2}\bigg)
    \leqslant CK^{2}\zeta\|B\|_{q}^{4}\leqslant Cc_{0}.
\end{aligned}
\end{equation}
The last line follows from \eqref{r32}, $s\leqslant c_{1}\zeta\|B\|_{0}^{-2}$ and \eqref{r35}. By \eqref{r32}, \eqref{dn} and \eqref{fr}, taking $c_{0}$ sufficiently small ensures
\begin{equation}\label{hc}
    0<\lambda_{s}\leqslant\|b_{0}-sb_{1}\|<1,\qquad \|F_{s}\|_{0}\leqslant\frac{\lambda_{s}}{20}.
\end{equation}
Thus \eqref{hc} implies uniform hyperbolicity of
$\big(\alpha,\operatorname{diag}(e^{\lambda_{s}},e^{-\lambda_{s}})+F_{s}\big)$.
By \eqref{pe}, \eqref{av} and \eqref{dl}, the cocycle $(\alpha,S_{E-s}^{V})$ is uniformly hyperbolic for every $0<s\leqslant s_{*}$.

Johnson's theorem \cite{Johnson198654} now gives \eqref{r33}. Since $E\in\Sigma_{V,\alpha}$ and $0<N_{V,\alpha}(E)<1$, the interval in \eqref{r33} lies in a bounded spectral gap with right endpoint $E$. Its length is at least $s_{*}=c_{1}\zeta\|B\|_{0}^{-2}$, proving \eqref{r34}.
\end{proof}

\begin{remark}\label{r37}
Suppose that \eqref{lp} holds with $\zeta$ replaced by $z$, where $0<|z|<\frac{1}{2}$ and $K^{2}|z|\|B\|_{q}^{4}\leqslant c_{0}$. For the energy displacement $s\operatorname{sgn}(z)$, the matrices in \eqref{bp}, with $\zeta=z$, satisfy
\begin{equation*}
    \det(b_{0}-s\operatorname{sgn}(z)b_{1})=-s|z|a+s^{2}\Delta-\frac{1}{4}s^{2}z^{2}a^{2}.
\end{equation*}
The estimates \eqref{r35}, \eqref{dh} and \eqref{fr}, with $\zeta$ replaced by $|z|$, therefore give
\begin{equation*}
    \bigg\{E-s\operatorname{sgn}(z):0<s\leqslant c_{1}|z|\|B\|_{0}^{-2}\bigg\}\cap\Sigma_{V,\alpha}=\varnothing.
\end{equation*}
If $E=E_{m}^{+}$ and $z<0$, continuity of the IDS would extend its level set to the right of $E_{m}^{+}$, contradicting the definition of $E_{m}^{+}$. Hence $z>0$ at $E=E_{m}^{+}$.
\end{remark}

\section{Construction of Schr\"odinger Operators with Polynomially Decaying Spectral Gaps}
Fix $\delta\in(0,1)$ and $k>361\tau$ as in \hyperref[GYGJMGYQ]{Theorem~\ref*{GYGJMGYQ}}(2). 
In this section, we introduce some definitions as follows
\begin{equation}
\label{r13} \varepsilon_{m}=m^{-21\tau}, \quad N_{l_{j}}=l_{j+1}, \quad l_{j}={M_{0}}^{{(\frac{3}{2})}^{j}}, \quad h_{j}={\frac{17\tau}{2\pi}}l_{j}^{-1}\ln l_{j},
\end{equation}
where   ${M_{0}}$ is sufficiently large  fixed constant . For every positive odd integer $j$, choose
\begin{equation*}
n_{j}=(\lceil1.95l_{j}\rceil,0,\ldots,0),\qquad \mathcal{K}=\{n_{j}:j\geqslant1\text{ is odd}\}.
\end{equation*}
Then $1.9l_{j}<|n_{j}|<2l_{j}$ for all these indices. In particular, $\mathcal{K}$ is infinite and no two adjacent shells contain elements of $\mathcal{K}$. This choice satisfies
\begin{equation*}
\begin{aligned}
\#\{n\in\mathcal{K}:2l_{j-2}< |n|\leqslant\ 2l_{j}\}\leqslant1, \\
\#\{n \in \mathcal{K}: 2l_{j-1}<|n|\leqslant 1.9l_{j}\}=0.
\end{aligned}
\end{equation*}
Denote
\begin{equation*}
\begin{aligned}
& \mathcal{A}_{j} \coloneqq\{n \in \mathbb{Z}^{d}|\ \ 2l_{j-1} \leqslant |n|< 2l_{j}\},  j\geqslant2, \\
& \mathcal{A}_{1} \coloneqq\{n \in \mathbb{Z}^{d}|\ \ 0 \leqslant |n| < 2l_{1}\},
\end{aligned}
\end{equation*}
and
\begin{equation*}
\tilde{h}_{j}=\begin{cases} h_{j}, \quad \text { as } \mathcal{K} \cap \mathcal{A}_{j} \neq \emptyset \\
\frac{3}{4} h_{j-1}, \quad \text { as } \mathcal{K} \cap \mathcal{A}_{j}=\emptyset.
\end{cases}
\end{equation*}
Moreover, set
\begin{equation}\label{pot}
V(\theta)\coloneqq\sum_{n\in\mathcal{K}}|n|^{-(1+\frac{\delta}{2})k}\cos(2\pi\langle n,\theta\rangle).
\end{equation} 
The series converges in $C^{k}$, since the factor $(2\pi)^{q}$ from every derivative of order $q\leqslant k$ is included in a constant depending only on $k$ and $d$, and
\begin{equation*}
\sum_{n\in\mathcal{K}}|n|^{k}|n|^{-{(1+\frac{\delta}{2})k}} =\sum_{n\in\mathcal{K}}|n|^{-\frac{\delta k}{2}}<\infty.
\end{equation*}
Moreover, $\|V\|_{k}$ can be made arbitrarily small by increasing ${M_{0}}$.
Set $q_{0}\coloneqq\lfloor3\tau+d\rfloor+1$ and take $K$ from \eqref{dt} with $q=q_{0}$. Let $c_{0},c_{1}$ be the constants in \hyperref[r30]{Lemma~\ref*{r30}} at this regularity. Choose $M_{0}$ sufficiently large so that the iteration conditions below hold and
\begin{equation}\label{is}
\begin{aligned}
    \|V\|_{k}&\leqslant\varepsilon_{0},\qquad K^{2}(1.9l_{1})^{147\tau-(1+\frac{\delta}{2})k}\leqslant c_{0},\\
    (1.9l_{1})^{122\tau-(1+\frac{\delta}{2})k}&\leqslant\frac{1}{4},\qquad c_{1}(1.9l_{1})^{\frac{\delta k}{2}+\frac{\tau}{10}}\geqslant1.
\end{aligned}
\end{equation}
These conditions can be imposed simultaneously since $k>361\tau$ and $\delta>0$.
Then the cocycle $(\alpha, S_{E}^{V}(\theta))$ of the Schr\"odinger operator $H_{V, \alpha, \theta}$ can be rewritten as 
\begin{equation*}
 S_{E}^{V}(\theta)\coloneqq \begin{pmatrix} E-V(\theta) & -1\\1 & 0
\end{pmatrix}
=A_{1}\mathrm{e}^{V(\theta)W} =A_{1}\prod_{n_{j} \in \mathcal{K} }\mathrm{e}^{V_{j}(\theta)W},
\end{equation*}
where $W=\begin{pmatrix}
0 & 0 \\1 & 0
\end{pmatrix}, 
A_{1}=\begin{pmatrix}
E & -1 \\1 & 0
\end{pmatrix}$ and
\begin{equation*}
V_{j}=\begin{cases} |n_{j}|^{-{(1+\frac{\delta}{2})k}}\cos{(2\pi\langle n_{j},\theta\rangle)}, \ \text { if } \{n_{j}\} \coloneqq \mathcal{K} \cap \mathcal{A}_{j},  \\
0, \quad \text { if } \mathcal{K} \cap \mathcal{A}_{j}=\emptyset.
\end{cases}
\end{equation*}
From now on we work in $\mathrm{SU}(1,1)$ coordinates and, to keep the notation, replace $A_{1}$ and $W$ by
\begin{equation*}
A_{1}=M\begin{pmatrix}E&-1\\1&0\end{pmatrix}M^{-1},\qquad W=M\begin{pmatrix}0&0\\1&0\end{pmatrix}M^{-1} =\frac{1}{2}\begin{pmatrix}-i&-i\\i&i\end{pmatrix}.
\end{equation*}
For a constant $A_{j}\in\mathrm{SU}(1,1)$, we write
$\rho_{j}=-\rho(\alpha,M^{-1}A_{j}M)$, with representatives chosen locally and with rotation numbers measured in turns. Thus an elliptic $A_{j}$ has the diagonal normal form $\operatorname{diag}(e^{{2\pi}i\rho_{j}},e^{-{2\pi}i\rho_{j}})$. In the resonance conditions below, ${\operatorname{dist}}(x,{\mathbb{Z}})$ denotes distance in $\mathbb{R}/\mathbb{Z}$.
For given $\eta>0$, $\alpha\in\mathbb{R}^{d}$ and $A\in\mathrm{SU}(1,1)$, we define a decomposition $C_{h}^{\omega}(\mathbb{T}^{d},\mathrm{su}(1,1))=C_{h}^{nre}(\eta)\oplus C_{h}^{re}(\eta)$ such that, for every $Y\in C_{h}^{nre}(\eta)$,
\begin{equation*}
    A^{-1}Y(\cdot+\alpha)A-Y(\cdot)\in C_{h}^{nre}(\eta), \qquad |A^{-1}Y(\cdot+\alpha)A-Y(\cdot)|_{h}\geqslant\eta|Y|_{h}.
\end{equation*}
Let $\mathbb{P}_{nre}$ and $\mathbb{P}_{re}$ be the standard projections from $C_{h}^{\omega}(\mathbb{T}^{d},\mathrm{su}(1,1))$ onto $C_{h}^{nre}(\eta)$ and $C_{h}^{re}(\eta)$, respectively.

\begin{theorem}\label{GJKDXL}
    Let $\alpha\in \mathrm{DC}_{d}(\gamma, \tau)$, $\gamma>0$, $\tau>d$. Suppose that ${k}>361\tau$, $A_{j}\in\mathrm{SU}(1,1)$ with
\begin{equation*}
\|A_{j}\|\leqslant3\prod_{r=1}^{j-1}(1+12l_{r}^{118\tau-{(1+\frac{\delta}{2})k}})\leqslant4,
\end{equation*}
Let $B^{(j)}\in C^{\omega}_{h_{j}}(2\mathbb{T}^{d}$, $\mathrm{SU}(1, 1))$ and $f_{j} \in C^{\omega}_{\tilde{h}_{j}}(\mathbb{T}^{d}$, $\mathrm{su}(1, 1))$ satisfying $|f_{j}|_{\tilde{h}_{j}}\leqslant \varepsilon_{l_{j}}^{{(2+\delta)k}}$, $|B^{(j)}|_{h_{j}}\leqslant \varepsilon_{l_{j}}^{-2}$, and suppose that $\operatorname{Ad}(B^{(j)})$ is $1$-periodic in every coordinate, as is the case for the conjugacies constructed below. Then there exist $B_{j} \in C^{\omega}_{h_{j}}(2\mathbb{T}^{d}$, $\mathrm{SU}(1, 1))$, ${f_{j+1}}\in C^{\omega}_{\tilde{h}_{j+1}}(\mathbb{T}^{d},\mathrm{su}(1, 1))$ and $A_{j+1} \in \mathrm{SU}(1, 1)$ such that 
    \begin{equation}
\label{r14} B_{j}(\theta+\alpha)(A_{j}e^{F_{j}(\theta)})B^{-1}_{j}(\theta)=A_{j+1}e^{F_{j+1}(\theta)},
\end{equation}
    where $F_{j}\in C^{k}(\mathbb{T}^{d},\mathrm{su}(1,1))$ is defined  by \begin{equation*}
e^{F_{j}}=e^{f_{j}}\prod \limits_{p\geqslant j} e^{\operatorname{Ad}(B^{(j)})(V_{p} W)},
\end{equation*} $B^{(1)}=I$ and $B^{(j+1)}=B_{j}B^{(j)}$ with estimates $|f_{j+1}|_{\tilde{h}_{j+1}}\leqslant\varepsilon_{l_{j+1}}^{{(2+\delta)k}}$,   $|B^{(j+1)}|_{h_{j+1}}\leqslant \varepsilon_{l_{j+1}}^{-2}$.
Furthermore, we obtain the following estimates for the two cases separately.

    \textbf{(Non-resonant case)} if for any $n\in \mathbb{Z}^{d}$ with $0<|n|\leqslant N_{l_{j}}$, we have \\${\operatorname{dist}}(2 \rho_{j}-\langle n, \alpha \rangle,{\mathbb{Z}})\geqslant \varepsilon_{l_{j}}^{\frac{1}{7}}$, then
    \begin{equation}
\label{n-r-c}
\begin{split}
|B_{j}-I|_{\tilde{h}_{j}} &\leqslant 12{l_{j}}^{127\tau-{(1+\frac{\delta}{2})k}}, \\
\|A_{j+1}-A_{j}\| &\leqslant 12\|A_{j}\|{l_{j}}^{118\tau-{(1+\frac{\delta}{2})k}},
\end{split}
\end{equation}
     and
     \begin{equation}
\label{10}
\begin{cases}
\|A_{j+1}-A_{j}-A_{j}\langle \operatorname{Ad}(B^{(j)}) V_{j}(\theta) W\rangle\| \leqslant30\|A_{j}\|l_{j}^{299\tau-{(2+\delta)k}}, \quad \mathcal{K} \cap \mathcal{A}_{j} \neq \emptyset,\\
\|A_{j+1}-A_{j}\| \leqslant30\|A_{j}\|l_{j}^{299\tau-{(2+\delta)k}}, \quad \mathcal{K} \cap \mathcal{A}_{j}=\emptyset .
\end{cases}
\end{equation}
\textbf{(Resonant case)} if there exists $n^{*}_{j} \in \mathbb{Z}^{d}$ with $0<|n^{*}_{j}|\leqslant N_{l_{j}}$ such that \\${\operatorname{dist}}(2 \rho_{j}-\langle n^{*}_{j}, \alpha \rangle,{\mathbb{Z}})<\varepsilon_{l_{j}}^{\frac{1}{7}}, $ then $B_{j}\coloneqq Q_{n^{*}_{j}}e^{Y_{j}}P_{j}, $ where \begin{equation*}
Q_{n^{*}_{j}}=\operatorname{diag}(e^{-{\pi}i\langle n^{*}_{j},\theta\rangle},e^{{\pi}i\langle n^{*}_{j},\theta\rangle}),
\end{equation*} $P_{j} \in \mathrm{SU}(1, 1)$, $Q_{n_{j}^{*}}\in C^{\omega}_{h_{j}}(2\mathbb{T}^{d},\mathrm{SU}(1,1))$ and $e^{Y_{j}}\in C^{\omega}_{\tilde{h}_{j}}(\mathbb{T}^{d}, \mathrm{SU}(1, 1)) $ with the following estimates 
\begin{equation*}
|e^{Y_{j}}-I|_{\tilde{h}_{j}}\leqslant 36{l_{j}}^{124\tau-{(1+\frac{\delta}{2})k}}, \qquad
\|P_{j}\|^{2}\leqslant 3\varepsilon_{l_{j}}^{-\frac{1}{7}}.
\end{equation*}
Moreover, (j+1)-th step is non-resonant. And if we rewrite  $A_{j+1}=(\begin{array}{cc}*& {b_{j+1}} \\ \overline{b_{j+1}} & *\end{array})$ and $P_{j} \operatorname{Ad}(B^{(j)}) \cdot(V_{j} W) P_{j}^{-1}=(\begin{array}{cc}* & T_{j} \\ \overline{T_{j}} & *\end{array})$, we can get
\begin{equation}
\label{12}
\begin{cases}
|b_{j+1}e^{-{2\pi}i(\rho_{j}-\frac{\langle n_{j}^{*}, \alpha \rangle}{2})}-\widehat{T}_{j}(n^{*}_{j})| \leqslant 324l_{j}^{263\tau-{(2+\delta)k}}, \quad \mathcal{K} \cap \mathcal{A}_{j} \neq \emptyset,\\
|b_{j+1}| \leqslant 324l_{j}^{263\tau-{(2+\delta)k}}, \quad \mathcal{K} \cap \mathcal{A}_{j}=\emptyset .
\end{cases}
\end{equation}
\end{theorem}
\begin{proof}
Set 
\begin{equation*}
\begin{aligned}
e^{\tilde{F}_{j}(\theta)}=e^{f_{j}(\theta)}e^{\operatorname{Ad}(B^{(j)})V_{j}(\theta)W}.
\end{aligned}
\end{equation*}
 If $\mathcal{K}\cap\mathcal{A}_{j}\neq\emptyset$, then \eqref{r13} yields $|V_{j}(\theta)|_{h_{j}} \leqslant l_{j}^{34\tau-{(1+\frac{\delta}{2})k}} $, together with the hypotheses on $f_{j}$ stated in the \hyperref[GJKDXL]{Theorem~\ref*{GJKDXL}} we have
\begin{equation}
\label{r16}
\begin{aligned}
|\tilde{F}_{j}|_{\tilde{h}_{j}}\leqslant 3{l_{j}}^{118\tau-{(1+\frac{\delta}{2})k}}.
\end{aligned}
\end{equation}
And if $\mathcal{K} \cap \mathcal{A}_{j}= \emptyset$, then $\tilde{F}_{j}=f_{j}$ and \eqref{r16} still holds.

\textbf{Non-resonant case:} We define
\begin{equation*}
\begin{aligned}
\Lambda_{j}=\{f\in C_{\tilde{h}_{j}}^{\omega}(\mathbb{T}^{d}, \mathrm{su}(1, 1))| ~f(\theta)=\sum_{n\in \mathbb{Z}^{d}, 0<|n|\leqslant N_{l_{j}}}\widehat{f}(n)e^{{2\pi}i\langle n, \theta \rangle}\}.
\end{aligned}
\end{equation*}

For any $0<|n|\leqslant N_{l_{j}}$, we have
\begin{equation*}
\begin{aligned}
{\operatorname{dist}}(\langle n, \alpha\rangle,{\mathbb{Z}})\geqslant\varepsilon_{l_{j}}^{\frac{1}{7}}.
\end{aligned}
\end{equation*}
By the Diophantine and non-resonance conditions, together with $\|A_{j}\|\leqslant4$,  for any $Y\in\Lambda_{j},$ we have
\begin{equation*}
\begin{aligned}
A_{j}^{-1}Y_{j}(\theta+\alpha)A_{j}-Y_{j}(\theta)\in C_{{\tilde{h}_{j}}}^{nre}(\varepsilon^{\frac{3}{7}}_{l_{j}}),
\end{aligned}
\end{equation*}
and
\begin{equation*}
\begin{aligned}
|A_{j}^{-1}Y_{j}(\theta+\alpha)A_{j}-Y_{j}(\theta)|_{\tilde{h}_{j}}\geqslant \varepsilon_{l_{j}}^{\frac{3}{7}}|Y|_{\tilde{h}_{j}}.
\end{aligned}
\end{equation*}	
Thus $\Lambda_{j}\subseteq C_{{\tilde{h}_{j}}}^{nre}
(\varepsilon_{l_{j}}^{\frac{3}{7}}).$

We recall the following lemma from \cite{HeCheng2025}, which will be used to remove the non-resonant terms.

\begin{lemma}\label{remove}
   Assume that $A \in \mathrm{SU}(1, 1)$ and $\eta \leqslant (2|A|)^{-30}$. Then, for any $F \in C_{h}^{\omega}(\mathbb{T}^{d}, \mathrm{su}(1, 1))$ with $|F|_{h} < \eta^{61/30}$, there exists $Y \in C_{h}^{nre}(\eta)$, $F_{*} \in C_{h}^{re}(\eta)$ such that
\begin{equation*}
e^{Y(\theta+\alpha)} (A e^{F(\theta)}) e^{-Y(\theta)} = A e^{F_{*}(\theta)},
\end{equation*}
with estimates
\begin{equation*}
|Y|_{h} \leqslant 2\eta^{-1}|F|_{h}, \quad |F_{*} - \mathbb{P}_{re}F|_{h} \leqslant 2\eta^{-61/10}|F|_{h}^{2}.
\end{equation*}
\end{lemma}

For $\ell_1$ large enough , one can check $\varepsilon_{l_{j}}^{\frac{3}{7}}\leqslant(2\|{A_{j}}\|)^{-30}$ and $|\tilde{F}_{j}|_{\tilde{h}_{j}}\leqslant(\varepsilon_{l_{j}}^{\frac{3}{7}})^{\frac{61}{30}}$, then by \hyperref[remove]{Lemma~\ref*{remove}} , there exist $Y_{j}\in C^{nre}_{\tilde{h}_{j}}(\varepsilon_{l_{j}}^{\frac{3}{7}})$ and $F^{*}_{j}\in C^{re}_{\tilde{h}_{j}}(\varepsilon_{l_{j}}^{\frac{3}{7}})$  such that
\begin{equation*}
\begin{aligned}
e^{Y_{j}(\theta+\alpha)}(A_{{j}}e^{\tilde{F}_{j}(\theta)})e^{-Y_{j}(\theta)}=A_{{j}} e^{F^{*}_{j}(\theta)},
\end{aligned}
\end{equation*}
with the estimates
\begin{equation}
\label{r17}
\begin{split}
|Y_{j}|_{\tilde{h}_{j}} \leqslant2\varepsilon_{l_{j}}^{-\frac{3}{7}}|\tilde{F}_{j}|_{\tilde{h}_{j}} \leqslant6{l_{j}}^{127\tau-{(1+\frac{\delta}{2})k}}, \\
|F^{*}_{j}-\mathbb{P}_{re}\tilde{F}_{j}|_{\tilde{h}_{j}} \leqslant18{l_{j}}^{299\tau-{(2+\delta)k}}.
\end{split}
\end{equation}
Combined with ${k>361\tau}$, we obtain that
\begin{equation}
\label{r18}
\begin{aligned}
|F^{*}_{j}|_{\tilde{h}_{j}} \leqslant 2 |\tilde{F}_{j}|_{\tilde{h}_{j}}\leqslant 6{l_{j}}^{118\tau-{(1+\frac{\delta}{2})k}}.
\end{aligned}
\end{equation}
Let $B_{j}=e^{Y_{j}}$, then we have
\begin{equation}
\label{resonantB_j}
\begin{aligned}
|B_{j}-I|_{\tilde{h}_{j}} \leqslant2|Y_{j}|_{\tilde{h}_{j}}\leqslant 12{l_{j}}^{127\tau-{(1+\frac{\delta}{2})k}},
\end{aligned}
\end{equation}
which together with the estimates about $B^{(j)}$ in \hyperref[GJKDXL]{Theorem~\ref*{GJKDXL}} yield
\begin{equation*}
\begin{aligned}
|B^{(j+1)}|_{h_{j+1}}\leqslant|B^{(j)}|_{h_{j}}|B_{j}|_{h_{j}}\leqslant2\varepsilon_{l_{j}}^{-2}\leqslant\varepsilon_{l_{j+1}}^{-2}.
\end{aligned}
\end{equation*}

Note that  $\Lambda_{j}\subseteq C_{\tilde{h}_{j}}^{nre}
(\varepsilon_{l_{j}}^{\frac{3}{7}})$, then for $F^{*}_{j}\in C^{re}_{\tilde{h}_{j}}(\varepsilon_{l_{j}}^{\frac{3}{7}})$,  we have
\begin{equation*}
\begin{aligned}
F^{*}_{j}=\langle F^{*}_{j}\rangle+\mathcal{R}_{N_{l_{j}}}F^{*}_{j}.
\end{aligned}
\end{equation*}
Let
\begin{equation*}
\begin{aligned}
e^{F^{*}_{j}}\eqqcolon e^{\langle F^{*}_{j}\rangle}e^{f_{j+1}},
\end{aligned}
\end{equation*}
with estimate
\begin{equation*}
\begin{aligned}
|f_{j+1}|_{\tilde{h}_{j+1}} \leqslant 2|\mathcal{R}_{N_{l_{j}}}F^{*}_{j}|_{\frac{3}{4}h_{j}}\leqslant 2e^{-{\frac{\pi}{2}}h_{j}N_{l_{j}}}|F_{j}^{*}|_{\tilde{h}_{j}}<\varepsilon^{{(2+\delta)k}}_{l_{j+1}}.
\end{aligned}
\end{equation*}
Let $A_{j+1}\coloneqq A_{j} e^{\langle F^{*}_{j}\rangle}$ and combining with \eqref{r18}, we have
\begin{equation*}
\begin{aligned}
\|A_{j+1}-A_{j}\| \leqslant 2|F_{j}^{*}|_{\tilde{h}_{j}}\|A_{j}\|\leqslant12\|A_{j}\|{l_{j}}^{118\tau-{(1+\frac{\delta}{2})k}},
\end{aligned}
\end{equation*}
and if we expand $A_{j+1}$ to higher order, we have
\begin{equation*}
\begin{aligned}
\|A_{j}+A_{j}\langle F^{*}_{j}\rangle-A_{j+1}\| \leqslant  \|A_{j}\||F^{*}_{j}|^{2}_{\tilde{h}_{j}} \leqslant{36}\|A_{j}\|l_{j}^{236\tau-{(2+\delta)k}}.
\end{aligned}
\end{equation*}
Moreover, \eqref{r17} yields
\begin{equation*}
\begin{aligned}
\|A_{j}\langle F^{*}_{j}\rangle-A_{j}\langle\tilde{F}_{j}\rangle\| \leqslant\|A_{j}\||F^{*}_{j}-\mathbb{P}_{re}\tilde{F}_{j}|_{\tilde{h}_{j}} \leqslant18\|A_{j}\|{l_{j}}^{299\tau-{(2+\delta)k}},
\end{aligned}
\end{equation*}
and
\begin{equation*}
\begin{aligned}
\|A_{j}\langle\tilde{F}_{j}(\theta)\rangle-A_{j}\langle \operatorname{Ad}(B^{(j)}) V_{j}(\theta) W\rangle\|\leqslant 2\|A_{j}\||f_{j}|_{\tilde{h}_{j}}\leqslant 2\|A_{j}\|\varepsilon_{l_{j}}^{{(2+\delta)k}}.
\end{aligned}
\end{equation*}

Combined with the above formulas and the condition ${k}>200\tau$, we arrive at
\begin{equation*}
\begin{aligned}
\|A_{j+1}-A_{j}-A_{j} \langle \operatorname{Ad}(B^{(j)}) V_{j}(\theta) W \rangle | \leqslant30\|A_{j}\|l_{j}^{299\tau-{(2+\delta)k}}.
\end{aligned}
\end{equation*}

\textbf{Resonant case:}
assume that there exists $n^{'}_{j}\neq n^{*}_{j}$ satisfying 
${\operatorname{dist}}(2 \rho_{j}-\langle n_{j}^{\prime}, \alpha\rangle,{\mathbb{Z}})<\varepsilon_{l_{j}}^{\frac{1}{7}}$, then by the Diophantine condition, we have
\begin{equation*}
\gamma|n_{j}^{\prime} - n^{*}_{j}|^{-\tau} \leqslant{\operatorname{dist}}(\langle n_{j}^{\prime} - n^{*}_{j}, \alpha\rangle,{\mathbb{Z}}) <2 \varepsilon_{l_{j}}^{\frac{1}{7}},
\end{equation*}
which implies that \begin{equation*}
|n_{j}^{\prime}|> \tilde{N}_{l_{j}} \gg  N_{l_{j}},
\end{equation*}
where
\begin{equation*}
\tilde{N}_{l_{j}}=(2^{-1}\gamma)^{\tau^{-1}} \varepsilon_{l_{j}}^{-\frac{1}{7\tau}}-N_{l_{j}}.
\end{equation*}
Therefore, $n^{*}_{j}$ is the unique resonant point.

In the resonant normal forms below, we use compatible local representatives of the angular quantities. The resonant condition and $\alpha\in\mathrm{DC}_{d}(\gamma,\tau)$ yield
\begin{equation*}
\begin{aligned}
{\operatorname{dist}}(2\rho_{j},{\mathbb{Z}}) \geqslant{\operatorname{dist}}(\langle n^{*}_{j}, \alpha\rangle,{\mathbb{Z}})- \varepsilon^{\frac{1}{7}}_{l_{j}} \geqslant 4 \varepsilon^{\frac{1}{7}}_{l_{j}}.
\end{aligned}
\end{equation*}
By Lemma 8.1 of \cite{HouY12}, we can find $P_{j} \in \mathrm{SU}(1, 1)$ such that
\begin{equation*}
A^{\prime}_{j} \coloneqq P_{j} A_{j} P_{j}^{-1} =(\begin{array}{cc}e^{{2\pi}i \rho_{j}} & 0 \\ 0 & e^{-{2\pi}i \rho_{j}}\end{array})
\end{equation*}
\begin{equation}\label{pd}
    \|P_{j}\|^{2}\leqslant\frac{C}{|\sin(2\pi\rho_{j})|}\leqslant C\gamma^{-1}|n_{j}^{*}|^{\tau}\leqslant3\varepsilon_{l_{j}}^{-1/7}.
\end{equation}
Then, we can conjugate the cocycle $(\alpha, A_{j}e^{\tilde{F_{j}}})$ to $(\alpha, A_{j}^{\prime}e^{g_{j}})$ where $g_{j} =P_{j}\tilde{F_{j}}P^{-1}_{j}$ and with the following estimate
\begin{equation*}
\begin{aligned}
|g_{j}|_{\tilde{h}_{j}} \leqslant\|P_{j}\|^{2}|\tilde{F}_{j}|_{\tilde{h}_{j}} \leqslant9{l_{j}}^{121\tau-{(1+\frac{\delta}{2})k}}.
\end{aligned}
\end{equation*}

Denote ${C}_{{\tilde{h}_{j}}}^{nre}(\varepsilon_{l_{j}}^{\frac{1}{7}})$ by the set of all $Y=(\begin{array}{cc}i a(\theta) & y(\theta) \\ {\overline{y(\theta)}} & -i a(\theta)\end{array}) \in {C}^{\omega}_{{\tilde{h}_{j}}}(\mathbb{T}^{d}, \mathrm{su}(1, 1))$ with
\begin{equation*}
a(\theta)=\sum_{n \in \mathbb{Z}^{d}, 0<|n|\leqslant \tilde{N}_{l_{j}}} \widehat{a}(n) e^{{2\pi}i\langle n, \theta\rangle}, \quad y(\theta)=\sum_{n \in \mathbb{Z}^{d}, 0 \leqslant|n|\leqslant \tilde{N}_{l_{j}}, n \neq n^{*}_{j}} \widehat{y}(n) e^{{2\pi}i\langle n, \theta\rangle},
\end{equation*}
and by ${C}_{{\tilde{h}_{j}}}^{re}$ the set of all $Z \in C^{\omega}_{\tilde{h}_{j}}(\mathbb{T}^{d}, \mathrm{su}(1, 1))$ of the form
\begin{equation*}
Z(\theta)=(\begin{array}{cc} i a & 0 \\
0 & -i a
\end{array})+(\begin{array}{cc}
0 & \widehat{z}(n^{*}_{j})e^{{2\pi}i\langle n^{*}_{j}, \theta\rangle} \\
\overline{\widehat{z}}(n^{*}_{j}) e^{-{2\pi}i\langle n^{*}_{j}, \theta\rangle} & 0
\end{array})+\sum_{n \in \mathbb{Z}^{d}, |n|> \tilde{N}_{l_{j}}} \widehat{Z}(n) e^{{2\pi}i\langle n, \theta\rangle},
\end{equation*}
where $a \in \mathbb{R}$, $\widehat{z}(n^{*}_{j})\in\mathbb{C} $ and the Fourier coefficients of $Z$ satisfy the reality relations determined by $Z(\theta)\in\mathrm{su}(1,1)$ on the real torus.

For $\eta=\varepsilon_{l_{j}}^{1/7}$, 
applying \hyperref[remove]{Lemma~\ref*{remove}} to $(\alpha,A_{j}' e^{g_{j}(\theta)})$, there exist $Y_{j} \in C^{nre}_{\tilde{h}_{j}}$ and $g^{*}_{j} \in {C}_{\tilde{h}_{j}}^{r e}(\varepsilon_{l_{j}}^{\frac{1}{7}})$ such that
\begin{equation*}
\begin{aligned}
e^{Y_{j}(\theta+\alpha)}(A^{\prime}_{j} e^{g_{j}(\theta)}) e^{-Y_{j}(\theta)}=A^{\prime}_{j} e^{g^{*}_{j}(\theta)},
\end{aligned}
\end{equation*}
with
\begin{equation}
\label{two}
\begin{split}
|Y_{j}|_{\tilde{h}_{j}} \leqslant 2\varepsilon_{l_{j}}^{-\frac{1}{7}}|g_{j}|_{\tilde{h}_{j}}\leqslant 18{l_{j}}^{124\tau-{(1+\frac{\delta}{2})k}}, \\
|g^{*}_{j}-\mathbb{P}_{re}g_{j}|_{\tilde{h}_{j}} \leqslant 2 \varepsilon_{l_{j}}^{-\frac{61}{70}}|g_{j}|_{\tilde{h}_{j}}^{2} \leqslant 162{l_{j}}^{261\tau-{(2+\delta)k}} .
\end{split}
\end{equation}
Thus, we can get
\begin{equation}
\label{r19}
\begin{aligned}
|e^{Y_{j}}-I|_{\tilde{h}_{j}} \leqslant2|Y_{j}|_{\tilde{h}_{j}} \leqslant 36{l_{j}}^{124\tau-{(1+\frac{\delta}{2})k}}, \quad |g^{*}_{j}|_{\tilde{h}_{j}}\leqslant2|g_{j}|_{\tilde{h}_{j}} \leqslant 18{l_{j}}^{121\tau-{(1+\frac{\delta}{2})k}}.
\end{aligned}
\end{equation}
Furthermore, we can write $g^{*}_{j}(\theta) \in {C}_{\tilde{h}_{j}}^{re}$ in the following form
\begin{equation*}
\begin{aligned}
g^{*}_{j}(\theta)=(\begin{array}{cc}i a_{j} & b^{*}_{j} e^{{2\pi}i\langle  n^{*}_{j}, \theta\rangle} \\ \bar{b}^{*}_{j} e^{-{2\pi}i\langle n^{*}_{j}, \theta\rangle} & -i a_{j}\end{array})+\mathcal{R}_{\tilde{N}_{l_{j}}}g^{*}_{j},
\end{aligned}
\end{equation*}
where $a_{j} \in \mathbb{R}, b^{*}_{j}=(\widehat{{g}^{*}_{j}}(n_{j}^{*}))_{12} \in \mathbb{C}$ with the estimates
\begin{equation*}
\begin{aligned}
&|a_{j}| \leqslant |g^{*}_{j}|_{\tilde{h}_{j}}\leqslant 18{l_{j}}^{121\tau-{(1+\frac{\delta}{2})k}}, \quad|b^{*}_{j}| \leqslant  |g^{*}_{j}|_{\tilde{h}_{j}} e^{-{2\pi}|n^{*}_{j}| {\tilde{h}_{j}}}< 18{l_{j}}^{121\tau-{(1+\frac{\delta}{2})k}}.
\end{aligned}
\end{equation*}
Now we impose a rotation $Q_{n^{*}_{j}}(\theta)$ on the cocycle $(\alpha, A'_{j}e^{g^{*}_{j}(\theta)})$, then 
\begin{equation*}
\begin{aligned}
Q_{n^{*}_{j}}(\theta+\alpha)(A^{\prime}_{j} e^{g^{*}_{j}(\theta)}) Q_{n^{*}_{j}}^{-1}(\theta)=\tilde{A}_{j} e^{g^{\prime}_{j}},
\end{aligned}
\end{equation*}
where
\begin{equation*}
\begin{aligned}
\tilde{A}_{j}= Q_{n^{*}_{j}}(\theta+\alpha)A_{j}^{\prime}Q_{n^{*}_{j}}^{-1}(\theta)=&(\begin{array}{cc} e^{{2\pi}i(\rho_{j}-\frac{\langle n^{*}_{j}, \alpha\rangle}{2})} & 0 \\
0 & e^{-{2\pi}i(\rho_{j}-\frac{\langle n^{*}_{j}, \alpha\rangle}{2})}
\end{array}), \\
g^{\prime}_{j}= &Q_{n^{*}_{j}} g^{*}_{j}(\theta) Q_{n^{*}_{j}}^{-1}.
\end{aligned}
\end{equation*}
And we can write
\begin{equation*}
g^{\prime}_{j}=\tilde{C}_{j}  +\widetilde{g}_{j},
\end{equation*}
where $\tilde{C}_{j} \coloneqq(\begin{array}{cc}i a_{j} & b^{*}_{j} \\ \bar{b}^{*}_{j} & -i a_{j}\end{array})$, $ \widetilde{g}_{j} \coloneqq Q_{n^{*}_{j}} (\mathcal{R}_{{\tilde{N}_{l_{j}}}}g^{*}_{j}) Q_{n^{*}_{j}}^{-1} $.

Denote $ B_{j}=Q_{n^{*}_{j}}  e^{Y_{j}}  P_{j} $, $A_{j+1}=\tilde{A}_{j} e^{\tilde{C}_{j}}$, $ e^{f_{j+1}(\theta)}=e^{-\tilde{C}_{j}}e^{\tilde{C}_{j}+\widetilde{g}_{j}}$ and $spec(\tilde{C}_{j})=\pm \lambda_{j}$, 
 one can check that
\begin{equation*}
\begin{aligned}
B_{j}(\theta+\alpha)(A_{j} e^{\tilde{F}_{j}(\theta)}) B_{j}^{-1}(\theta)=A_{j+1} e^{f_{j+1}(\theta)},
\end{aligned}
\end{equation*}
together with \eqref{r19} and Baker-Campbell Theorem yield that
\begin{equation*}
\begin{aligned}
|f_{j+1}|_{\tilde{h}_{j+1}}\leqslant 2|Q_{n^{*}_{j}}|_{\frac{3}{4}h_{j}}^{2} \sum_{|n|>\tilde{N}_{l_{j}}}|g^{*}_{j}(n)| e^{{\frac{3\pi}{2}}|n|h_{j}}\leqslant 2e^{-{\pi}  N_{l_{j}} h_{j}}|g^{*}_{j}|_{\tilde{h}_{j}}<\varepsilon_{l_{j+1}}^{{(2+\delta)k}}.
\end{aligned}
\end{equation*}
 By  $M_{0}$ large enough, we also have $|g_{j}^{*}|_{\tilde{h}_{j+1}}\leqslant\varepsilon_{l_{j}}^{1/7}/2$. Let
$A_{j+1}=e^{C_{j+1}}$. Then
\begin{equation*}
\begin{aligned}
\|C_{j+1}\|\leqslant 2({2\pi}|\rho_{j}-\frac{\langle n^{*}_{j},\alpha\rangle}{2}|+|g_{j}^{*}|_{\tilde{h}_{j+1}})
\leqslant{8}\varepsilon_{l_{j}}^{1/7}.
\end{aligned}
\end{equation*}
and
\begin{equation*}
|\rho_{j+1}|{\leqslant\frac{\|C_{j+1}\|}{2\pi}}
<2\varepsilon_{l_{j}}^{1/7}.
\end{equation*}
For every $n\in\mathbb{Z}^{d}$ with $0<|n|\leqslant N_{l_{j+1}}$, we have
\begin{equation}\label{nr}
\operatorname{dist}\big(2\rho_{j+1}-\langle n,\alpha\rangle,\mathbb{Z}\big)\geqslant\operatorname{dist}\big(\langle n,\alpha\rangle,\mathbb{Z}\big)-2|\rho_{j+1}|>\varepsilon_{l_{j+1}}^{\frac{1}{7}},
\end{equation} 
which implies
(j+1)-th step is non-resonant.

By \eqref{nr}, every resonant step is followed by a non-resonant step. Hence, 
\begin{equation*}
\begin{aligned}
|B^{(j+1)}|_{h_{j+1}} \leqslant2|Q_{n_{j}^{*}}|_{h_{j+1}}\|P_{j}\||B^{(j)}|_{h_{j}} \leqslant4\sqrt3\,l_{j}^{56.25\tau} \leqslant\varepsilon_{l_{j+1}}^{-2}.
\end{aligned}
\end{equation*}
And
\begin{equation*}
|b_{j+1}e^{-{2\pi}i(\rho_{j}-\frac{\langle n_{j}^{*}, \alpha \rangle}{2})}-b_{j}^{*}|\leqslant \|\tilde{C}_{j}\|^{2}\leqslant |g_{j}^{*}|^{2}_{\tilde{h}_{j}}\leqslant 4\|P_{j}\|^{4}|\tilde{F}_{j}|^{2}_{\tilde{h}_{j}}.
\end{equation*}
The resonant projection preserves the upper-right Fourier coefficient at $n_{j}^{*}$. Hence \eqref{two} and $g_{j}=P_{j}\tilde{F}_{j}P_{j}^{-1}$ give
\begin{equation*}
\begin{aligned}
    |b_{j}^{*}-(\widehat{g}_{j}(n_{j}^{*}))_{12}|
    &\leqslant |g_{j}^{*}-\mathbb{P}_{re}g_{j}|_{\tilde{h}_{j}}e^{-2\pi|n_{j}^{*}|\tilde{h}_{j}}\\
    &\leqslant2\varepsilon_{l_{j}}^{-\frac{61}{70}}|\tilde{F}_{j}|_{\tilde{h}_{j}}^{2}\|P_{j}\|^{4}e^{-2\pi|n_{j}^{*}|\tilde{h}_{j}}.
\end{aligned}
\end{equation*}
and
\begin{equation*}
\begin{aligned}
| g_{j}(\theta)- P_{j}Ad(B^{(j)}) V_{j}(\theta) WP_{j}^{-1}|_{\tilde{h}_{j}}\leqslant 2|f_{j}|_{\tilde{h}_{j}}\|P_{j}\|^{2}.
\end{aligned}
\end{equation*}
Applying the triangle inequality to the above expression, we get
\begin{equation*}
\begin{cases}
|b_{j+1}e^{-{2\pi}i(\rho_{j}-\frac{\langle n_{j}^{*}, \alpha \rangle}{2})}-\widehat{T}_{j}(n^{*}_{j})| \leqslant 324l_{j}^{263\tau-{(2+\delta)k}}, \quad \mathcal{K} \cap \mathcal{A}_{j} \neq \emptyset,\\
|b_{j+1}| \leqslant 324l_{j}^{263\tau-{(2+\delta)k}}, \quad \mathcal{K} \cap \mathcal{A}_{j}=\emptyset .
\end{cases}
\end{equation*}
In the non-resonant case the stated bound for $\|A_{j+1}\|$ follows by multiplying the previous bound by $1+12l_{j}^{118\tau-{(1+\frac{\delta}{2})k}}$. In the resonant case $\|A_{j+1}\|\leqslant2$, so the same bound holds. Increasing ${M_{0}}$ makes the full product less than $4/3$. This also verifies the smallness condition involving $\|A_{j}\|$ in \hyperref[remove]{Lemma~\ref*{remove}} at every step. Conjugating the remaining real-torus tail gives \eqref{r14}, with
\begin{equation*}
e^{F_{j+1}}=e^{f_{j+1}}\prod_{p\geqslant j+1}e^{V_{p}\operatorname{Ad}(B^{(j+1)})W}.
\end{equation*}
\end{proof}

\section{\texorpdfstring{Proof of \hyperref[GYGJMGYQ]{Theorem~\ref*{GYGJMGYQ}}(2)}{Proof of Theorem 1.1(2)}}

In this section, we prove part (2) of \hyperref[GYGJMGYQ]{Theorem~\ref*{GYGJMGYQ}} for the potential defined in \eqref{pot}. The main step is the following quantitative reducibility result at the gap endpoints.

\begin{lemma}\label{r26}
Let $n_{J}\in\mathcal{K}$, and set $m_{J}\coloneqq-n_{J}$, $n\coloneqq|n_{J}|$ and $E\coloneqq E_{m_{J}}^{+}$. For every integer $0\leqslant q<(1+\frac{\delta}{2})k-127\tau$, there exist $B\in C^{q}(2\mathbb{T}^{d},\mathrm{SL}(2,\mathbb{R}))$ and $\zeta>0$ such that
\begin{equation}\label{r27}
    B(\theta+\alpha)^{-1}S_{E}^{V}(\theta)B(\theta)=\begin{pmatrix}1&\zeta\\0&1\end{pmatrix}.
\end{equation}
The following estimates hold:
\begin{equation}\label{r28}
\begin{aligned}
    \|B\|_{0}&\leqslant n^{\frac{6}{5}\tau}, \qquad \|B\|_{q}\leqslant n^{q+\frac{6}{5}\tau},\\
    n^{-(1+\frac{\delta}{2})k-\frac{7}{2}\tau}&\leqslant\zeta\leqslant n^{122\tau-(1+\frac{\delta}{2})k}.
\end{aligned}
\end{equation}
\end{lemma}

We postpone the proof of \hyperref[r26]{Lemma~\ref*{r26}} to the end of this section and first deduce part (2) of \hyperref[GYGJMGYQ]{Theorem~\ref*{GYGJMGYQ}} from the lower bound in \hyperref[r30]{Lemma~\ref*{r30}}.

\begin{proof}[{Proof of part (2) of {\hyperref[GYGJMGYQ]{Theorem~\ref*{GYGJMGYQ}}}}]
For $n_{J}\in\mathcal{K}$, set $m_{J}\coloneqq-n_{J}$ and $n\coloneqq|n_{J}|$. Apply \hyperref[r26]{Lemma~\ref*{r26}} with $q=q_{0}$. Since $3\tau+d<q_{0}<5\tau<(1+\frac{\delta}{2})k-127\tau$, by \eqref{is} and \eqref{r28},
\begin{equation}\label{gs}
    K^{2}\zeta\|B\|_{q_{0}}^{4}\leqslant K^{2}n^{147\tau-(1+\frac{\delta}{2})k}\leqslant c_{0},\qquad 0<\zeta\leqslant\frac{1}{4}.
\end{equation}
The label $m_{J}\neq0$ gives $0<N_{V,\alpha}(E_{m_{J}}^{+})<1$. Thus \eqref{r27} and \eqref{gs} verify the hypotheses of \hyperref[r30]{Lemma~\ref*{r30}}, and \eqref{r34} and \eqref{r28} yield
\begin{equation*}
\begin{aligned}
    |G_{m_{J}}(V)|&\geqslant c_{1}\zeta\|B\|_{0}^{-2}\geqslant c_{1}n^{-(1+\frac{\delta}{2})k-\frac{59}{10}\tau}\\
    &=\bigg(c_{1}n^{\frac{\delta k}{2}+\frac{\tau}{10}}\bigg)n^{-(1+\delta)k-6\tau}\geqslant n^{-(1+\delta)k-6\tau}.
\end{aligned}
\end{equation*}
The last inequality follows from \eqref{is}. For $M_{0}$ large enough,  $\|V\|_{k}\leqslant\varepsilon_{0}$. Since $\mathcal{K}$ is infinite and $|n_{J}|\to\infty$, the labels $m_{J}=-n_{J}$ give the required sequence.
\end{proof}

Before giving the proof of \hyperref[r26]{Lemma~\ref*{r26}}, we first introduce some useful notations and related lemmata.

 Set
\begin{equation*}
\begin{aligned}
W_{j}&=\operatorname{Ad}(B^{(j)})W=(\begin{array}{cc} i u_{j} & w_{j} e^{-{2\pi}i\langle \tilde{n}^{(j)}, \theta\rangle} \\
\overline{w_{j}} e^{{2\pi}i\langle\tilde{n}^{(j)}, \theta\rangle} & -i u_{j}
\end{array}), \\
\mathcal{Z}_{j} &\coloneqq|\langle w_{j}\rangle|, \quad \mathcal{M}_{j} \coloneqq|w_{j}|_{h_{j}}+|u_{j}|_{h_{j}}, \\
\mathcal{X}_{j} &\coloneqq \sup _{n \in \mathbb{Z}^{d}, |n| \geqslant|\widetilde{n}^{(j)}|} {\max\{|\widehat{w}_{j}(n)|,|\widehat{u}_{j}(n)|\}}, \\
\tilde{n}^{(1)}&=0, \quad \tilde{n}^{(j+1)}=\begin{cases} \tilde{n}^{(j)}, \quad  \text{$j$-th step is non-resonant}, \\
\tilde{n}^{(j)}+{n}^{*}_{j}, \quad  \text{$j$-th step is resonant}.
\end{cases}
\end{aligned}
\end{equation*}
Fix $n_{J}\in\mathcal{K}$ and a spectral energy $E$ with
\begin{equation*}
2\rho(\alpha,S_{E}^{V})\equiv-\langle n_{J},\alpha\rangle\pmod{{\mathbb{Z}}}.
\end{equation*}
The corresponding gap label is $m_{J}=-n_{J}$. Since $\|V\|_{0}$ is small and $E\in\Sigma_{V,\alpha}$, the initial matrix satisfies $\|A_{1}\|<3$. We start with $f_{1}=0$ and $B^{(1)}=I$. By increasing ${M_{0}}$, all initial smallness conditions in \hyperref[GJKDXL]{Theorem~\ref*{GJKDXL}} hold.
Inductively, there exists $B^{(j)}\in C_{h_{j}}^{\omega}(2\mathbb{T}^{d},\mathrm{SU}(1,1))$ conjugating $(\alpha,MS_{E}^{V}M^{-1})$ to
$(\alpha,A_{j}e^{f_{j}}\prod_{p\geqslant j}e^{V_{p}\operatorname{Ad}(B^{(j)})W})$.
{Write $\tilde{\rho}_{j}$ for the negative of the fibered rotation number of this cocycle in real coordinates.} Then
\begin{equation*}
 \tilde{\rho}_{j}\equiv\frac{1}{2}\langle n_{J}-\widetilde{n}^{(j)},\alpha\rangle\pmod{{\frac{1}{2}\mathbb{Z}}}.
\end{equation*}
By \eqref{rho_holder}, using $\|A_{j}\|\leqslant4$ from \hyperref[GJKDXL]{Theorem~\ref*{GJKDXL}}, we obtain
 
\begin{equation}
\label{rho-rho}
\begin{aligned}
|\rho_{j}-\tilde{\rho}_{j}|&\leqslant {C}\|A_{j}-A_{j}e^{f_{j}}\prod_{p\geqslant j}e^{V_{p}W_{j}}\|_{0}^{\frac{1}{2}}\\
&\leqslant {C}\|A_{j}\|^{1/2} {\biggl(\varepsilon_{l_{j}}^{{(2+\delta)k}}+|W_{j}|_{h_{j}} \sum_{n_{p}\in\mathcal{K},\ p\geqslant j}|n_{p}|^{-{(1+\frac{\delta}{2})k}}\biggr)^{1/2}}\\
&\leqslant l_{j}^{43\tau-{\frac{(2+\delta)k}{4}}}< \varepsilon_{l_{j}}^{\frac{1}{6}}.
\end{aligned}
\end{equation}
 Thus, we arrive that
\begin{equation}
\label{rhoooo}
\begin{aligned}
{\operatorname{dist}}(2\rho_{j}-\langle n_{J}-\tilde{n}^{(j)}, \alpha \rangle,{\mathbb{Z}})<\varepsilon_{l_{j}}^{\frac{1}{7}}.
\end{aligned}
\end{equation}

\begin{lemma}\label{n-range}
For all $j\geqslant 1$, we have 
\begin{equation}
\label{23} |\tilde{n}^{(j)}|\leqslant \frac{3}{2}l_{j}, \quad e^{{2\pi}|\tilde{n}^{(j)}|h_{j}}\leqslant l_{j}^{25.5\tau}.
\end{equation}
If $j>1$ and $j$-th step is resonant, we have
\begin{equation}
|\tilde{n}^{(j)}|\leqslant \frac{3}{2}l_{j-1},\quad e^{{2\pi}|\tilde{n}^{(j)}|h_{j}}\leqslant {l_{j}}^{\frac{\tau}{10}},\label{j}
\end{equation}
\begin{equation}
|n_{j}^{*}|>l_{j},\quad e^{-{2\pi}|n_{j}^{*}|h_{j}}\leqslant l_{j}^{-17\tau},\label{r20}
\end{equation}
\begin{equation}
|\tilde{n}^{(j+1)}|\geqslant \frac{99}{100}l_{j},\quad e^{-{2\pi}|\tilde{n}^{(j+1)}|h_{j}}\leqslant l_{j}^{-16.83\tau}.\label{r21}
\end{equation}
\end{lemma}
\begin{proof}
  By definition, \eqref{23} holds for $j=1$. Assume that \eqref{23} is true for $j$. Then if $j$-th step is non-resonant, \eqref{23} is still holds since $\tilde{n}^{(j+1)}=\tilde{n}^{(j)}$. And if $j$-th step is resonant, we have 
    \begin{equation*}
|\tilde{n}^{(j+1)}|\leqslant |\tilde{n}^{(j)}|+|n_{j}^{*}|\leqslant \frac{3}{2}l_{j}+l_{j+1}\leqslant \frac{3}{2}l_{j+1}.
\end{equation*}
If step $j>1$ is resonant, step $j-1$ is non-resonant by \hyperref[GJKDXL]{Theorem~\ref*{GJKDXL}}. Thus the definition of $\tilde{n}^{(j+1)}$ and \eqref{23} at step $j-1$ give \eqref{j}.
    
    Next, we assume that $j$-th step is resonant. If $|n_{j}^{*}|\leqslant l_{j}=N_{l_{j-1}}$, then due to the properties of rotation numbers, we obtain that 
    \begin{equation*}
|\rho_{j}-\rho_{j-1}|\leqslant C\|A_{j}-A_{j-1}\|^{\frac{1}{2}}\leqslant {l_{j-1}}^{60\tau-{\frac{(2+\delta)k}{4}}}.
\end{equation*}
    From the inequality above together with the resonant condition, it follows that
    \begin{equation*}
\begin{aligned}
{\operatorname{dist}}(2\rho_{j-1}-\langle n_{j}^{*}, \alpha \rangle,{\mathbb{Z}})<\varepsilon_{l_{j-1}}^{\frac{1}{7}},
\end{aligned}
\end{equation*}
    which contradicts with $(j-1)$-th step is non-resonant. By simple calculation, \eqref{r20} can be readily deduced.

    The estimate \eqref{j} and \eqref{r20} imply that $|\tilde{n}^{(j)}|\leqslant \frac{|n_{j}^{*}|}{100}$. Combined with the triangle inequality, this yields \eqref{r21}.
\end{proof}

   \begin{lemma}\label{6.5}
We have $J$-th step is resonant and $n^{*}_{J}=n_{J}-\tilde{n}^{(J)}$. Moreover for $j\geqslant J+1$, $j$-th step is non-resonant and $\tilde{\rho_{j}}=0$.
\end{lemma}
\begin{proof}
    In the case $J=1$, by \eqref{rhoooo}, we arrive that
    \begin{equation*}
\begin{aligned}
{\operatorname{dist}}(2\rho_{1}-\langle n_{1}, \alpha \rangle,{\mathbb{Z}})<\varepsilon_{l_{1}}^{\frac{1}{7}},
\end{aligned}
\end{equation*}
    which means the first step is resonant since $0<|n_{1}|\leqslant N_{l_{1}}$.

For $J>1$, \hyperref[n-range]{Lemma~\ref*{n-range}} gives
    \begin{equation*}
\begin{aligned}
N_{l_{J}}&\geqslant |n_{J}-\tilde{n}^{(J)}|\geqslant 1.9l_{J} -\frac{3}{2}l_{J} >0,
\end{aligned}
\end{equation*}
    together with \eqref{rhoooo} imply that $J$-th step is resonant and $n^{*}_{J}=n_{J}-\tilde{n}^{(J)}$.
    Then $\tilde{\rho}_{J+1}=\frac{1}{2} \langle n_{J}-\tilde{n}^{(J+1)}, \alpha \rangle=0$ and $(J+1)$-th step is non-resonant since two adjacent steps cannot be resonant simultaneously.
    Moreover $\tilde{n}^{(J+2)}=\tilde{n}^{(J+1)}$ and $\tilde{\rho}_{J+2}=\tilde{\rho}_{J+1}=0$.

Next, we assume that $j$-th step is non-resonant and $\tilde{\rho}_{j+1}=0$ for some $j\geqslant J+1$ and then we consider $(j+1)$-step. By \eqref{rho-rho} and Diophantine condition, we have
\begin{equation*}
\begin{aligned}
&{\operatorname{dist}}(2\rho_{j+1}-\langle n,\alpha\rangle,{\mathbb{Z}})\\
&\quad\geqslant\frac{\gamma}{N_{l_{j+1}}^{\tau}} -2l_{j+1}^{43\tau-{\frac{(2+\delta)k}{4}}} >\varepsilon_{l_{j+1}}^{1/7},\qquad0<|n|\leqslant N_{l_{j+1}}.
\end{aligned}
\end{equation*}
which deduces that $(j+1)$-th step is non-resonant , $\tilde{n}^{(j+2)}=\tilde{n}^{(j+1)}$ and $\tilde{\rho}_{j+2}=\tilde{\rho}_{j+1}=0$.
\end{proof}
By \hyperref[6.5]{Lemma~\ref*{6.5}}, step $J$ is resonant. Let $j_{0}\leqslant J$ denote the first resonant step. We next estimate $\mathcal{M}_{j}$, $\mathcal{X}_{j}$ and $\mathcal{Z}_{j}$.
\begin{lemma}
    \begin{equation}
\mathcal{M}_{{j_{0}}+1}\leqslant 24\varepsilon_{l_{j_{0}}}^{-\frac{1}{7}},\label{Mj01}
\end{equation}
\begin{equation}
\mathcal{Z}_{{j_{0}}+1}\geqslant \frac{1}{100} \varepsilon_{l_{j_{0}}}^{\frac{1}{7}},\label{Zj01}
\end{equation}
\begin{equation}
\mathcal{X}_{{j_{0}}+1} \leqslant 24 \varepsilon_{l_{j_{0}}}^{\frac{2}{3}}\label{Xj01}.
\end{equation}
    \end{lemma}
    \begin{proof}
   For $1\leqslant j\leqslant j_{0}$, we have
\begin{equation}\label{wpre}
    |W_{j}|_{h_{j}}\leqslant2, \qquad \|[A_{j},\langle W_{j}\rangle]\|\geqslant\frac{1}{2}.
\end{equation} 
From the definition of $j_{0}$ and \eqref{resonantB_j}, the former inequality holds trivially. Moreover, by simple calculation, we have $\|[A_{1},\langle W_{1}\rangle]\|\geqslant1$. Note that $W_{j+1}=\operatorname{Ad}(B_{j})W_{j}$, together with \eqref{resonantB_j}, we obtain
    \begin{equation*}
\begin{aligned}
|W_{j+1}-W_{j}|_{h_{j}}&\leqslant 2|B_{j}^{-1}|_{h_{j}}|B_{j}-I|_{h_{j}}|W_{j}|_{h_{j}}\\
&\leqslant 96l_{j}^{127\tau-{(1+\frac{\delta}{2})k}},
\end{aligned}
\end{equation*}
    and together with the second inequality in \eqref{n-r-c} and
    \begin{equation*}
\begin{aligned}
\|[A_{j}, \langle W_{j} \rangle]-[A_{1}, \langle W_{1} \rangle]\|\leqslant 2(\|A_{j}-A_{1}\|\|W_{j}\|+\|A_{1}\|\|W_{j}-W_{1}\|),
\end{aligned}
\end{equation*}
    this gives $\|[A_{j}, \langle W_{j} \rangle]\|\geqslant \frac{1}{2}$.
By the first estimate in \eqref{r19}, we have
    \begin{equation*}
\begin{aligned}
|\operatorname{Ad}(e^{Y_{j_{0}}}P_{j_{0}})W_{j_{0}}-\operatorname{Ad}(P_{j_{0}})W_{j_{0}}|_{h_{j_{0}}}&\leqslant 2|e^{-Y_{j_{0}}}|_{h_{j_{0}}}|e^{Y_{j_{0}}}-I|_{h_{j_{0}}}\|P_{j_{0}}\|^{2}|W_{j_{0}}|_{h_{j_{0}}}\\
&\leqslant 144{l_{j_{0}}}^{124\tau-{(1+\frac{\delta}{2})k}}\|P_{j_{0}}\|^{2}|W_{j_{0}}|_{h_{j_{0}}}.
\end{aligned}
\end{equation*}
    Then by the triangle inequality, we obtain that
    \begin{equation*}
\begin{aligned}
\mathcal{M}_{{j_{0}}+1}&\leqslant 2|\operatorname{Ad}(e^{Y_{j_{0}}}P_{j_{0}})W_{j_{0}}|_{h_{j_{0}}}\\
&\leqslant 2(|\operatorname{Ad}(P_{j_{0}})W_{j_{0}}|_{h_{j_{0}}}+|\operatorname{Ad}(e^{Y_{j_{0}}}P_{j_{0}})W_{j_{0}}-\operatorname{Ad}(P_{j_{0}})W_{j_{0}}|_{h_{j_{0}}})\\
&\leqslant 2\|P_{j_{0}}\|^{2}|W_{j_{0}}|_{h_{j_{0}}}(1+144{l_{j_{0}}}^{124\tau-{(1+\frac{\delta}{2})k}})\\
&\leqslant 8\|P_{j_{0}}\|^{2} \leqslant 24\varepsilon_{l_{j_{0}}}^{-\frac{1}{7}}.
\end{aligned}
\end{equation*}
As for $\mathcal{Z}_{{j_{0}}+1}$, \hyperref[lemma2.1]{Lemma~\ref*{lemma2.1}} and \hyperref[lemma2.2]{Lemma~\ref*{lemma2.2}} yield
    \begin{equation*}
\begin{aligned}
\mathcal{Z}_{{j_{0}}+1}&\geqslant \frac{1}{2}{|\sin({2\pi}\rho_{j_{0}})|}^{-1}\|P_{j_{0}}\|^{-2}\|[A_{j_{0}}, \langle W_{j_{0}} \rangle]\|-4|e^{Y_{j_{0}}}-I|_{h_{j_{0}}}\|P_{j_{0}}\|^{2}|W_{j_{0}}|_{h_{j_{0}}}\\
&\geqslant \frac{1}{100}\varepsilon_{l_{j_{0}}}^{\frac{1}{7}}.
\end{aligned}
\end{equation*}
 In the following, we give the estimate of $\mathcal{X}_{{j_{0}}+1}$. If $j_{0}=1$, then 
    \begin{equation*}
\begin{aligned}
\mathcal{X}_{2}&=\sup_{|n|\geqslant |n_{1}^{*}|}\max\{|\widehat{w}_{2}(n)|,|\widehat{u}_{2}(n)|\}\\
&\leqslant |\operatorname{Ad}(e^{Y_{1}}P_{1})(W)-\langle \operatorname{Ad}(e^{Y_{1}}P_{1})(W) \rangle|_{h_{1}}\\
&\leqslant 4|e^{Y_{1}}-I|_{h_{1}}\|P_{1}\|^{2}\\
&\leqslant 432 l_{1}^{127\tau-{(1+\frac{\delta}{2})k}}\leqslant 24 \varepsilon_{l_{1}}^{\frac{2}{3}}.
\end{aligned}
\end{equation*}
    The last inequality holds since ${k}>200\tau$. And if $j_{0}>1$, then
    \begin{equation*}
\begin{aligned}
\mathcal{X}_{{j_{0}}+1}&\leqslant (|w_{{j_{0}}+1}|_{h_{j_{0}}}+|u_{{j_{0}}+1}|_{h_{j_{0}}})e^{-{2\pi}|\tilde{n}^{({j_{0}}+1)}|h_{j_{0}}}\\
&\leqslant {24\varepsilon_{l_{j_{0}}}^{-1/7}}e^{-{2\pi}|n_{j_{0}}^{*}|h_{j_{0}}}\\
&\leqslant 24 \varepsilon_{l_{j_{0}}}^{\frac{2}{3}}.
\end{aligned}
\end{equation*}
\end{proof}

\begin{lemma}
    For $j_{0}+1\leqslant j \leqslant J$, if j-th step is non-resonant, then 
    \begin{equation}
\label{r23}
\begin{aligned}
\mathcal{M}_{j+1}&\leqslant \mathcal{M}_{j}+{200}l_{j}^{{178}\tau-{(1+\frac{\delta}{2})k}}\mathcal{M}_{j},\\
\mathcal{X}_{j+1}&\leqslant \mathcal{X}_{j}+{200}l_{j}^{{178}\tau-{(1+\frac{\delta}{2})k}}\mathcal{M}_{j},\\
\mathcal{Z}_{j+1}&\geqslant \mathcal{Z}_{j}-{200}l_{j}^{{178}\tau-{(1+\frac{\delta}{2})k}}\mathcal{M}_{j}.
\end{aligned}
\end{equation}
    And if j-th step is resonant, then
    \begin{equation}
\label{r24}
\begin{aligned}
\mathcal{M}_{j+1}&\leqslant {24 l_{j}^{3.2\tau}}\mathcal{M}_{j},\\
\mathcal{X}_{j+1}&\leqslant {24 l_{j}^{-13.63\tau}}\mathcal{M}_{j},\\
\mathcal{Z}_{j+1}&\geqslant \mathcal{Z}_{j}-3\mathcal{X}_{j}-{864}l_{j}^{127\tau-{(1+\frac{\delta}{2})k}}\mathcal{M}_{j}.
\end{aligned}
\end{equation}
\end{lemma}
\begin{proof}
{\textbf{Non-resonant case.} Put $m=\widetilde{n}^{(j)}$. Since $\widetilde{n}^{(j+1)}=m$, then we have
\begin{equation*}
\begin{aligned}
|u_{j+1}-u_{j}|_{h_{j}}+|w_{j+1}-w_{j}|_{h_{j}} &\leqslant2e^{{2\pi}|m|h_{j}}|W_{j+1}-W_{j}|_{h_{j}}\\
&\leqslant16|e^{Y_{j}}-I|_{h_{j}}e^{{4\pi}|m|h_{j}}\mathcal{M}_{j}\\
&\leqslant{200}l_{j}^{{178}\tau-{(1+\frac{\delta}{2})k}}\mathcal{M}_{j}.
\end{aligned}
\end{equation*}
Restriction to $h_{j+1}$ proves the upper bound for $\mathcal{M}_{j+1}$. The same estimate bounds the change in every Fourier coefficient, and therefore proves the bounds for $\mathcal{X}_{j+1}$ and $\mathcal{Z}_{j+1}$.

\textbf{Resonant case.} The two adjacent steps cannot both be resonant, so \eqref{j} gives $e^{{2\pi}|m|h_{j}}\leqslant l_{j}^{\tau/10}$. After removing the phase introduced by $Q_{n_{j}^{*}}$, the functions $u_{j+1}$ and $w_{j+1}$ extend to the wider strip $h_{j}$, and satisfy
\begin{equation*}
|u_{j+1}|_{h_{j}}+|w_{j+1}|_{h_{j}} \leqslant8\|P_{j}\|^{2}e^{{4\pi}|m|h_{j}}\mathcal{M}_{j} \leqslant24l_{j}^{3.2\tau}\mathcal{M}_{j}.
\end{equation*}
This proves the bound for $\mathcal{M}_{j+1}$, and Fourier decay on this wider strip gives
\begin{equation*}
\mathcal{X}_{j+1}\leqslant24l_{j}^{3.2\tau}\mathcal{M}_{j} e^{-{2\pi}|\widetilde{n}^{(j+1)}|h_{j}} \leqslant24l_{j}^{-13.63\tau}\mathcal{M}_{j}.
\end{equation*}

Apply \hyperref[lemma2.1]{Lemma~\ref*{lemma2.1}} to the upper-right coefficient at $-m$. The main term involves $\langle w_{j}\rangle$, and the other two coefficients are bounded by $\mathcal{X}_{j}$. Thus we have
\begin{equation*}
    8|e^{Y_{j}}-I|_{h_{j}}\|P_{j}\|^{2}\mathcal{M}_{j}\leqslant864l_{j}^{127\tau-(1+\frac{\delta}{2})k}\mathcal{M}_{j}.
\end{equation*}
Together, these estimates yield the last inequality in \eqref{r24}.}
\end{proof}

\begin{lemma}\label{lemma6.4}
    For all $j_{0}+1\leqslant j \leqslant J$, we have
    \begin{equation}
\mathcal{M}_{j} \leqslant \varepsilon_{l_{j}}^{-\frac{1}{3}} ,\label{Mj}
\end{equation}
\begin{equation}
\mathcal{Z}_{j} \geqslant 10\mathcal{X}_{j}+\frac{1}{2^{j}}\mathcal{Z}_{j_{0}+1}.\label{Zj}
\end{equation}
\end{lemma}
\begin{proof}
    \eqref{Mj} holds trivally for $j=j_{0}+1$ since \eqref{Mj01}. Next we assume \eqref{Mj} holds for $p\geqslant j_{0}+1$, then we need to prove the case $p+1$. If $p$-th step is non-resonant, by the first inequality in \eqref{r23}, it follows that
    \begin{equation*}
\begin{aligned}
\mathcal{M}_{p+1}\leqslant (1+{200}l_{p}^{{178}\tau-{(1+\frac{\delta}{2})k}})\mathcal{M}_{p} \leqslant \varepsilon_{l_{p+1}}^{-\frac{1}{3}}.
\end{aligned}
\end{equation*} 
    If $p$-th step is resonant, from \eqref{r24} and our previous assumptions on $\mathcal{M}_{p}$, we have
    \begin{equation*}
\begin{aligned}
\mathcal{M}_{p+1}\leqslant {24 l_{p}^{3.2\tau}} \mathcal{M}_{p} \leqslant {24 l_{p}^{3.2\tau}} \varepsilon_{l_{p}}^{-\frac{1}{3}} \leqslant \varepsilon_{l_{p+1}}^{-\frac{1}{3}}.
\end{aligned}
\end{equation*}
By \eqref{Zj01}, it follows that 
    \begin{equation}
\label{160t}
\begin{aligned}
l_{p}^{{186}\tau-{(1+\frac{\delta}{2})k}}\leqslant\frac{1}{100}\varepsilon_{l_{p}}^{\frac{1}{7}} \cdot 100l_{p}^{{189}\tau-{(1+\frac{\delta}{2})k}}\leqslant (11\cdot 2^{p+1})^{-1}\mathcal{Z}_{j_{0}+1},
\end{aligned}
\end{equation}
    where the last inequality holds since ${M_{0}}$ is sufficiently large.
    
    As for $\mathcal{Z}_{j}$, by \eqref{Zj01} and \eqref{Xj01}, \eqref{Zj} is true for $j=j_{0}+1$. Assume it holds for $p\geqslant j_{0}+1$, then we consider the case $p+1$. If $p$-th step is non-resonant, by \eqref{r23} and \eqref{Mj}, we have 
    \begin{equation*}
\begin{aligned}
\mathcal{Z}_{p+1}&\geqslant \mathcal{Z}_{p}-{200}l_{p}^{{178}\tau-{(1+\frac{\delta}{2})k}}\mathcal{M}_{p}\geqslant \mathcal{Z}_{p}-l_{p}^{{186}\tau-{(1+\frac{\delta}{2})k}},\\
\mathcal{X}_{p+1} &\leqslant \mathcal{X}_{p}+{200}l_{p}^{{178}\tau-{(1+\frac{\delta}{2})k}}\mathcal{M}_{p} \leqslant \mathcal{X}_{p}+l_{p}^{{186}\tau-{(1+\frac{\delta}{2})k}}.
\end{aligned}
\end{equation*}
Combining the above inequality with the earlier assumption $\mathcal{Z}_{p} \geqslant 10\mathcal{X}_{p}+\frac{1}{2^{p}}\mathcal{Z}_{j_{0}+1}$, we obtain 
    \begin{equation*}
\begin{aligned}
\mathcal{Z}_{p+1}&\geqslant 10\mathcal{X}_{p}+\frac{1}{2^{p}}\mathcal{Z}_{j_{0}+1}-l_{p}^{{186}\tau-{(1+\frac{\delta}{2})k}}\\
&\geqslant 10\mathcal{X}_{p+1}+\frac{1}{2^{p}}\mathcal{Z}_{j_{0}+1}-11l_{p}^{{186}\tau-{(1+\frac{\delta}{2})k}}\\
&\geqslant 10\mathcal{X}_{p+1}+\frac{1}{2^{p+1}}\mathcal{Z}_{j_{0}+1}.
\end{aligned}
\end{equation*}

    Otherwise, if $p$-th step is resonant, by \eqref{r24} and the estimate of $\mathcal{M}_{p}$ we arrive that
    \begin{equation*}
\begin{aligned}
\mathcal{X}_{p+1}\leqslant {24 l_{p}^{-13.63\tau}}\mathcal{M}_{p} \leqslant {24 l_{p}^{-6.63\tau}} \leqslant (11\cdot 2^{p+1})^{-1} \mathcal{Z}_{j_{0}+1}.
\end{aligned}
\end{equation*}
    Combining the above inequality with \eqref{160t}, we have
    \begin{equation*}
\begin{aligned}
\mathcal{Z}_{p+1}&\geqslant \mathcal{Z}_{p}-3\mathcal{X}_{p}-{864}l_{p}^{127\tau-{(1+\frac{\delta}{2})k}}\mathcal{M}_{p}\\
&\geqslant \mathcal{Z}_{p}-3\mathcal{X}_{p}-l_{p}^{{186}\tau-{(1+\frac{\delta}{2})k}}\\
&\geqslant 7\mathcal{X}_{p}+\frac{1}{2^{p}}\mathcal{Z}_{j_{0}+1}-(11\cdot 2^{p+1})^{-1}\mathcal{Z}_{j_{0}+1}\\
&\geqslant 10\mathcal{X}_{p+1}+7\mathcal{X}_{p}+(\frac{12}{11\cdot 2^{p+1}}-\frac{1}{11\cdot 2^{p+1}})\mathcal{Z}_{j_{0}+1}\\
&\geqslant 10\mathcal{X}_{p+1}+\frac{1}{2^{p+1}}\mathcal{Z}_{j_{0}+1}.
\end{aligned}
\end{equation*}
    \end{proof}

\begin{lemma}
Write $A_{j}=\begin{pmatrix}a_{j}&b_{j}\\\overline{b_{j}}&\overline{a_{j}}\end{pmatrix}$. Then $|b_{j}|\geqslant|n_{J}|^{-(1+\frac{\delta}{2})k-3.5\tau}$ for every $j\geqslant J+1$.
\end{lemma}

\begin{proof}

By \hyperref[6.5]{Lemma~\ref*{6.5}}, step $J$ is resonant and $n_{J}^{*}=n_{J}-\widetilde{n}^{(J)}$. By the resonant case of \hyperref[GJKDXL]{Theorem~\ref*{GJKDXL}}, there exists $P_{J}\in\mathrm{SU}(1,1)$ such that $P_{J}A_{J}P_{J}^{-1}=\operatorname{diag}(e^{2\pi i\rho_{J}},e^{-2\pi i\rho_{J}})$. Set
\begin{equation}\label{tildeW}
    \widetilde{W}_{J}\coloneqq\operatorname{Ad}(P_{J})W_{J}
    =\begin{pmatrix}
        i\widetilde{u}_{J}&\widetilde{w}_{J}e^{-2\pi i\langle\widetilde{n}^{(J)},\theta\rangle}\\
        \overline{\widetilde{w}_{J}}e^{2\pi i\langle\widetilde{n}^{(J)},\theta\rangle}&-i\widetilde{u}_{J}
    \end{pmatrix}.
\end{equation}
Let $T_{J}$ be the upper-right entry of $V_{J}\widetilde{W}_{J}$. The Fourier expansion of $V_{J}$ gives
\begin{equation}\label{r25}
    \widehat{T}_{J}(n_{J}^{*})=\frac{\langle\widetilde{w}_{J}\rangle+\widehat{\widetilde{w}}_{J}(2n_{J})}{2|n_{J}|^{(1+\frac{\delta}{2})k}}.
\end{equation}
Suppose first that $J>j_{0}$. By \hyperref[lemma2.1]{Lemma~\ref*{lemma2.1}}, \hyperref[lemma6.4]{Lemma~\ref*{lemma6.4}} and \eqref{r25},
\begin{equation*}
\begin{aligned}
    |\widehat{T}_{J}(n_{J}^{*})|
    &\geqslant\frac{|(P_{J})_{11}|^{2}(\mathcal{Z}_{J}-3\mathcal{X}_{J})-\|P_{J}\|^{2}\mathcal{X}_{J}}{2|n_{J}|^{(1+\frac{\delta}{2})k}}\\
    &\geqslant\frac{3\mathcal{Z}_{J}}{20|n_{J}|^{(1+\frac{\delta}{2})k}}
    \geqslant|n_{J}|^{-(1+\frac{\delta}{2})k-3.2\tau}.
\end{aligned}
\end{equation*}
Here  we  used $|(P_{J})_{11}|\geqslant1$, $\|P_{J}\|^{2}\leqslant4|(P_{J})_{11}|^{2}$ and $\mathcal{Z}_{J}\geqslant10\mathcal{X}_{J}$. The last inequality follows from \eqref{Zj01} and \eqref{Zj}, which give $\mathcal{Z}_{J}\geqslant(100\cdot2^{J})^{-1}l_{J}^{-3\tau}.$

If $J=j_{0}$, then $\widetilde{n}^{(J)}=0$ and $n_{J}^{*}=n_{J}$. By \eqref{wpre}, the bound on $P_{J}$ in the resonant case of \hyperref[GJKDXL]{Theorem~\ref*{GJKDXL}}, and \hyperref[lemma2.1]{Lemma~\ref*{lemma2.1}}, we obtain
\begin{equation*}
|\langle\widetilde{w}_{J}\rangle|\geqslant\frac{1}{100}\varepsilon_{l_{J}}^{1/7}, \qquad |\widetilde{w}_{J}|_{h_{J}}\leqslant6\varepsilon_{l_{J}}^{-1/7}.
\end{equation*}
Thus, by  \eqref{r25},
\begin{equation*}
\begin{aligned}
    |\widehat{T}_{J}(n_{J}^{*})|
    &\geqslant\frac{1}{2|n_{J}|^{(1+\frac{\delta}{2})k}}
    \bigg(\frac{1}{100}\varepsilon_{l_{J}}^{1/7}-6\varepsilon_{l_{J}}^{-1/7}e^{-4\pi|n_{J}|h_{J}}\bigg)\\
    &\geqslant|n_{J}|^{-(1+\frac{\delta}{2})k-3.2\tau}.
\end{aligned}
\end{equation*}
In either case, \eqref{12} in \hyperref[GJKDXL]{Theorem~\ref*{GJKDXL}} yields
\begin{equation*}
\begin{aligned}
    |b_{J+1}|
    &\geqslant|\widehat{T}_{J}(n_{J}^{*})|-324l_{J}^{263\tau-(2+\delta)k}\\
    &\geqslant|n_{J}|^{-(1+\frac{\delta}{2})k-3.2\tau}-324l_{J}^{263\tau-(2+\delta)k}\\
    &\geqslant(1+2^{-J-1})|n_{J}|^{-(1+\frac{\delta}{2})k-3.5\tau}.
\end{aligned}
\end{equation*}
By \hyperref[6.5]{Lemma~\ref*{6.5}}, every step after $J$ is non-resonant. Since $\mathcal{K}\cap\mathcal{A}_{J+1}=\varnothing$ by construction, applying \eqref{10} at step $J+1$ and \eqref{n-r-c} thereafter gives
\begin{equation}\label{tail}
\begin{aligned}
    \sum_{j=J+1}^{\infty}\|A_{j+1}-A_{j}\|
    &\leqslant120l_{J+1}^{299\tau-(2+\delta)k}
    +48\sum_{j=J+2}^{\infty}l_{j}^{118\tau-(1+\frac{\delta}{2})k}\\
    &\leqslant2^{-J-1}|n_{J}|^{-(1+\frac{\delta}{2})k-3.5\tau}.
\end{aligned}
\end{equation}
The last inequality follows from $k>361\tau$, $l_{j+1}=l_{j}^{3/2}.$
Consequently, for every $j\geqslant J+1$,
\begin{equation}
\label{bl}
    |b_{j}|\geqslant|b_{J+1}|-\sum_{r=J+1}^{\infty}\|A_{r+1}-A_{r}\|
    \geqslant|n_{J}|^{-(1+\frac{\delta}{2})k-3.5\tau}.
\end{equation}
By \eqref{tail}, $A_{\infty}\coloneqq\lim_{j\to\infty}A_{j}$ exists, and its upper-right entry satisfies the same lower bound.
\end{proof}

We can now complete the proof of \hyperref[r26]{Lemma~\ref*{r26}}.

\begin{proof}[Proof of {\hyperref[r26]{Lemma~\ref*{r26}}}]
At every resonant step $r$, the resonance condition, the Diophantine condition and $|n_{r}^{*}|\leqslant l_{r+1}=l_{r}^{3/2}$ give
\begin{equation*}
    \operatorname{dist}(2\rho_{r},\mathbb{Z})\geqslant\gamma|n_{r}^{*}|^{-\tau}-l_{r}^{-3\tau}\geqslant\frac{\gamma}{2}|n_{r}^{*}|^{-\tau}.
\end{equation*}
 Applying \eqref{pd} at the resonant step $r$ with $\|A_{r}\|\leqslant4$, 
\begin{equation}\label{r29}
    \|P_{r}\|^{2}\leqslant C\frac{\|A_{r}\|}{|\sin(2\pi\rho_{r})|}\leqslant C\gamma^{-1}|n_{r}^{*}|^{\tau}.
\end{equation}
By \hyperref[6.5]{Lemma~\ref*{6.5}}, the last resonant step is $J$, with $n_{J}^{*}=n_{J}-\widetilde{n}^{(J)}$. By the resonant case of \hyperref[GJKDXL]{Theorem~\ref*{GJKDXL}}, the step following a resonant step is non-resonant. Consequently, the bounds on the resonant vectors give
\begin{equation}\label{rv}
\begin{aligned}
    |n_{J}^{*}|&\leqslant2n, \qquad
    \sum_{\substack{r\leqslant J\\r\text{ resonant}}}|n_{r}^{*}|\leqslant Cn,\\
    \sum_{\substack{r<J\\r\text{ resonant}}}\ln|n_{r}^{*}|
    &\leqslant\sum_{a=1}^{\infty}\bigg(\frac{2}{3}\bigg)^{2a-1}\ln l_{J}
    =\frac{6}{5}\ln l_{J}.
\end{aligned}
\end{equation}

Set $s_{J}\coloneqq(2\pi n)^{-1}$. By \eqref{r13}, $h_{J+1}<s_{J}<h_{J}$. Every factor in $B^{(J+1)}=B_{J}\cdots B_{1}$ is holomorphic on $\mathcal{S}_{h_{J}}$. Using $|Q_{n_{r}^{*}}|_{s_{J}}\leqslant2e^{\pi s_{J}|n_{r}^{*}|}$, together with \eqref{n-r-c}, \eqref{r19}, \eqref{r29} and \eqref{rv}, we obtain
\begin{equation}\label{bs}
    |B^{(J+1)}|_{s_{J}}\leqslant C^{J+1}(2n)^{\tau/2}l_{J}^{3\tau/5}\leqslant n^{\frac{23}{20}\tau}.
\end{equation}
The last inequality follows from $C^{J+1}2^{\tau/2}\leqslant n^{\tau/20}$, ensured uniformly in $J$ by the choice of $M_{0}$ and $\ln l_{J}=(3/2)^{J}\ln M_{0}$.

By \hyperref[6.5]{Lemma~\ref*{6.5}}, every step after $J$ is non-resonant. Thus \eqref{n-r-c} and \eqref{bs} give $|B^{(j)}|_{h_{j}}\leqslant2n^{\frac{23}{20}\tau}$ for $j\geqslant J+1$. Applying the Fourier Cauchy estimate, we obtain
\begin{equation}\label{bd}
\begin{aligned}
    \|B^{(j+1)}-B^{(j)}\|_{q}
    &\leqslant C_{q}h_{j}^{-q}|B^{(j+1)}-B^{(j)}|_{h_{j}}\\
    &\leqslant C_{q}n^{\frac{23}{20}\tau}l_{j}^{q+127\tau-(1+\frac{\delta}{2})k}.
\end{aligned}
\end{equation}
Since $q<(1+\frac{\delta}{2})k-127\tau$, \eqref{bd} implies that $B^{(j)}$ converges in $C^{q}$ to a projectively $\mathbb{T}^{d}$-periodic map $B^{(\infty)}$. By \eqref{bs}, \eqref{bd} and the Cauchy estimate on $\mathcal{S}_{s_{J}}$,
\begin{equation}\label{bq}
\begin{aligned}
    \|B^{(\infty)}\|_{q}
    &\leqslant C_{q}s_{J}^{-q}|B^{(J+1)}|_{s_{J}}
    +\sum_{j=J+1}^{\infty}\|B^{(j+1)}-B^{(j)}\|_{q}\\
    &\leqslant C_{q}n^{q+\frac{23}{20}\tau}
    \leqslant n^{q+\frac{6}{5}\tau}.
\end{aligned}
\end{equation}
Since only finitely many integers $q$ are admissible, the initial scale can be fixed so that these bounds hold simultaneously for all such $q$, independently of $J$.

By \eqref{tail}, $A_{\infty}\coloneqq\lim_{j\to\infty}A_{j}$ exists. Passing to the limit in the conjugacy identities \eqref{r14}, using the perturbation bound in \eqref{rho-rho}, gives
\begin{equation}\label{lc}
    B^{(\infty)}(\theta+\alpha)MS_{E}^{V}(\theta)M^{-1}B^{(\infty)}(\theta)^{-1}=A_{\infty}.
\end{equation}

Write $b_{\infty}\coloneqq(A_{\infty})_{12}$. The resonant normal form and \eqref{r19} give $|b_{J+1}|\leqslant Cl_{J}^{121\tau-(1+\frac{\delta}{2})k}$. Together with \eqref{tail} and \hyperref[bl]{Lemma~\ref*{bl}}, this yields
\begin{equation}\label{bi}
    n^{-(1+\frac{\delta}{2})k-\frac{7}{2}\tau}\leqslant|b_{\infty}|\leqslant Cl_{J}^{121\tau-(1+\frac{\delta}{2})k}.
\end{equation}
By \hyperref[6.5]{Lemma~\ref*{6.5}}, the reduced rotation number is zero. Since $E\in\Sigma_{V,\alpha}$, $A_{\infty}$ is not hyperbolic. The resonant normal form and \eqref{tail} also imply that $A_{\infty}$ is close to $I$. Together with \eqref{bi}, these facts show that $A_{\infty}$ is a nontrivial parabolic matrix.

Choose $R\in\mathrm{SO}(2,\mathbb{R})$ such that
\begin{equation}\label{sh}
    R^{-1}M^{-1}A_{\infty}MR=\begin{pmatrix}1&\zeta\\0&1\end{pmatrix}, \qquad |\zeta|=2|b_{\infty}|.
\end{equation}
Set $B(\theta)\coloneqq M^{-1}B^{(\infty)}(\theta)^{-1}MR$. Then \eqref{lc} and \eqref{sh} give \eqref{r27}, and $B$ is projectively $\mathbb{T}^{d}$-periodic. Since inversion on determinant-one matrices is the linear adjugate map in dimension two, $\|B\|_{q}=\|B^{(\infty)}\|_{q}$, so \eqref{bq} gives the required norm bounds. Moreover, \eqref{bi}, \eqref{sh} and $1.9l_{J}<n<2l_{J}$ imply
\begin{equation}\label{zb}
    n^{-(1+\frac{\delta}{2})k-\frac{7}{2}\tau}\leqslant|\zeta|\leqslant n^{122\tau-(1+\frac{\delta}{2})k}.
\end{equation}
Since $q_{0}<(1+\frac{\delta}{2})k-127\tau$, the same limiting map belongs to $C^{q_{0}}$ by \eqref{bd}. Thus \eqref{is}, \eqref{bq} and \eqref{zb} give
\begin{equation}\label{zs}
    K^{2}|\zeta|\|B\|_{q_{0}}^{4}\leqslant K^{2}n^{147\tau-(1+\frac{\delta}{2})k}\leqslant c_{0},\qquad 0<|\zeta|\leqslant\frac{1}{4}.
\end{equation}
Since $E=E_{m_{J}}^{+}$, \hyperref[r37]{Remark~\ref*{r37}} and \eqref{zs} imply $\zeta>0$.

 This completes the proof.
\end{proof}

\section*{Acknowledgments}
 Ao Cai was supported by the National Natural Science Foundation of China (Grant Nos. 12671227 and 12271091) and grant JSTJ-2025-674. Jiawei He was supported by the National Natural Science Foundation of China (Grant No. 12501247), the Natural Science Foundation of Fujian Province (Grant No. 2025J08232), the Startup Fund for Advanced Talents of Putian University (Grant No. 2023120), and the Fujian Alliance of Mathematics (Grant No. 2025SXLMMS07).

\section*{Statements and Declarations}
{\bf Conflict of Interest} 
The authors declare no conflicts of interest.
				
\vspace{0.2in}
{\bf Data Availability}
Data sharing is not applicable to this article as no new data were created or analyzed in this study.

\raggedbottom

\bibliographystyle{plain}
\bibliography{references}

\end{document}